\documentclass[12pt,a4paper]{article}
\usepackage[OT1]{fontenc}
\usepackage[latin1]{inputenc}
\usepackage[english]{babel}
\usepackage{amsfonts}
\usepackage{amsthm}
\newtheorem{teo}{Theorem}[section]
\newtheorem{prop}[teo]{Proposition}
\newtheorem{lem}[teo]{Lemma}
\newtheorem{coro}[teo]{Corollary}
\newtheorem{defi}[teo]{Definition}
\newtheorem{rem}[teo]{Remark}

\newtheorem{problem}[teo]{Problem}

\def\bh{{\cal B}({\cal H})}

\def\bp{{\cal B}_p({\cal H})}

\def\g{\mathfrak g}
\def\k{\mathfrak k}
\def\p{\mathfrak p}
\def\s{\mathfrak s}
\def\u{\mathfrak u}

\def\ad{{\rm{ad\, }}}

\begin{document}

\title{\vspace*{0cm}Manifolds of semi-negative curvature}
%53C45 Global surface theory (convex surfaces à la A. D. Aleksandrov)
%22E65 Infinite-dimensional Lie groups and their Lie algebras
%58B20 Riemannian, Finsler and other geometric structures
%53C40 Global submanifolds
%47B10 Operators belonging to operator ideals (nuclear, $p$-summing, in the Schatten-von Neumann classes, etc.)
\date{}
\author{Cristian Conde and Gabriel Larotonda\footnote{Both authors partially supported by Instituto Argentino de Matemática, CONICET}}

\maketitle

%\begin{spacing}{.95}
\abstract{\footnotesize{\noindent }The notion of nonpositive curvature in Alexandrov's sense is extended to include $p$-uniformly convex Banach spaces. Infinite dimensional manifolds of semi-negative curvature with a $p$-uniformly convex tangent norm fall in this class on nonpositively curved spaces, and several well-known results, such as existence and uniqueness of best approximations from convex closed sets, or the Bruhat-Tits fixed point theorem, are shown to hold in this setting, without dimension restrictions. Homogeneous spaces $G/K$ of Banach-Lie groups of semi-negative curvature are also studied, explicit estimates on the geodesic distance and sectional curvature are obtained. A characterization of convex homogeneous submanifolds is given in terms of the Banach-Lie algebras. A splitting theorem via convex expansive submanifolds is proven, inducing the corresponding splitting of the Banach-Lie group $G$. Finally, these notions are used to study the structure of the classical Banach-Lie groups of bounded linear operators acting on a Hilbert space, and the splittings induced by conditional expectations in such setting.}
%\end{spacing}

\setlength{\parindent}{0cm} %% para que no indente los parrafos nuevos

\section{Introduction}\label{intro1}

The study of nonpositively curved spaces began with the work of Hadamard in the early years of the last century, and the work of Cartan about twenty years later. But the foundations of the theory of metric spaces with upper curvature bounds were laid in the 50's with the work of Alexandrov \cite{alexandrov} and Busemann \cite{busemann}, who actually coined the term ``nonpositively curved space''. At the heart of their view-point (the use of conditions which are equivalent to nonpositive sectional curvature in the Riemannian case, rather than sectional curvature itself) is the work of Menger and Wald \cite{menger,wald}, who introduced the notions of curves in metric spaces, geodesic length spaces and comparison triangles. 

Nonpositive curvature in the sense of Alexandrov states that sufficiently small geodesic triangles in the metric space $(X,d)$ are at least as thin as corresponding Euclidean triangles. Equivalently $X$ verifies the CN inequality of Bruhat and Tits: for any $x\in X$ and any geodesic  segment $\gamma\in X$,
$$%referencia BruT72 en Bridson-Haefliger
\frac{1}{4}L(\gamma)^2 \leq \frac{1}{2}(d(x,\gamma_0)^2+d(x,\gamma_1)^2)-d(x,\gamma_{1/2})^2,
$$
provided $\gamma$ is sufficiently close to $x$. If $X$ is a 2-uniformly convex Banach space (in the sense of Ball, Carlen and Lieb \cite{ball}, i.e  there exists a positive constant $C$ such that
$$
2(\frac{1}{C}\|v\|^2+\|w\|^2)\leq \|v-w\|^2+\|v+w\|^2,
$$
for any $u,v\in X$), then the nonpositive curvature condition of Alexandrov holds for $X$ if $C\le 1$. It has been observed that Banach spaces with a $p$-uniformly convex norm ($p\ge 2$) share many of the nice properties of Hilbert spaces in spite of the fact that they do not verify Alexandrov's definition of nonpositive curvature: a Banach space has to be necessarily Euclidean  to verify the above inequality \cite[II.1.14]{bridson}. So it is natural to consider such Banach spaces as a convenient generalization of Euclidean space, leading us to introduce the notion of Busemann $p$-space, which is a geodesic length space which verifies a curvature condition %Rem 5.5 en Ballman
$$
\frac{1}{(2K)^p}L(\gamma)^p \leq \frac{1}{2}(d(x,\gamma_0)^p+d(x,\gamma_1)^p)-d(x,\gamma_{1/2})^p.
$$

The link with smooth manifolds is given by the following elementary fact: if $M$ is a Riemannian-Hilbert manifold of semi-negative sectional curvature, then
$$
\|(\exp_x)_{*v}(w)\|\ge \|w\|
$$
for any $x\in M$ and $v,w\in T_xX$ (here $(\exp_x)_*$ denotes the differential of the exponential map of $M$). This condition is adopted  in \cite{neeb}  by  Neeb as a definition of semi-negative curvature in the context of Banach-Finsler manifolds, and one of the main results in that paper is a Cartan-Hadamard theorem for such manifolds.  We show that if the Finsler norm of $M$ is $p$-uniformly convex, then $M$ can be regarded as a Busemann $p$-space.

Are these spaces nonpositively curved in the sense of Busemann? Namely, is the distance map between two geodesics a convex function in this setting? This question was shown to have a positive answer by Lawson and Lim, as part of their studies on symmetric spaces \cite{lim}. What other properties (of a Riemann-Hilbert manifold) can be translated to this context? e.g. existence of best approximations from convex sets, the Bruhat-Tits theorem for groups of isometries. One of the purposes of this paper is to answer some of the questions posed in Neeb's paper, assuming in some cases that the tangent norms of $M$ are $p$-uniformly convex, thus dealing with the Busemann $p$-spaces just introduced.

\medskip

In the special situation when $M=G/K$ is an homogeneous space of semi-negative curvature, it is obtained in \cite{neeb} a polar decomposition for $G$ that generalizes the usual polar decomposition for the group ${\cal B}({\cal H})^{\times}$ of invertible bounded operators in a Hilbert space ${\cal H}$. In this paper, we translate to this setting $M=G/K$ several results on operator theory -particularly results on the group of invertibles of $C^*$-algebras- that through time have been established using operator-theoretic techniques, such as the splitting theorems due to Porta and Recht \cite{pr}. To establish such results, we give  a detailed   characterization of the convex homogeneous submanifolds of $M$, which we think are interesting in their own right, since an infinite dimensional theory is still lacking.

This paper is organized as follows: in Section \ref{intro}, the reader can find the basic definitions concerning Banach-Finsler manifolds with spray, and an account on the results  in \cite{neeb}. In Section \ref{metricproblems}, we study manifolds $M$ of semi-negative curvature with a $p$-uniformly convex tangent norm, leading to the concept of Busemann $p$-space. We translate several results from the Riemannian context to this setting, and we establish metric splitting theorems for $M$ via convex submanifolds $C$ by means of the Birkhoff orthogonal to the tangent spaces $T_xC, x\in C$. Section \ref{homogeneous} establishes some general metric results on homogeneous spaces $M=G/K$ of semi-negative curvature, such as formulas for the geodesic distance and estimates of sectional curvature, and then a characterizations of the different levels of convexity that arise in this setting is given. We conclude with a splitting theorem for the homogeneous space $M$ via convex expansive reductive submanifolds, which gives the corresponding splitting of the Banach-Lie group $G$.

We have included an Appendix with homogeneous spaces constructed via operator algebras, generalizing the typical scheme of positive invertible operators, that is $G={\cal B}({\cal H})^{\times}$, $K$=the group of unitary operators of ${\cal H}$. This construction provides examples of manifolds of semi-negative curvature with a  $p$-uniformly convex tangent norm, and also of convex homogeneous submanifolds. Conditional expectations in ${\cal B}({\cal H})$ provide the splitting theorems, inducing factorization of linear operators via $C^*$-subalgebras of ${\cal B}({\cal H})$.

\section{Background}\label{intro}

Let $M$ be a Banach manifold with spray. Then $M$ is a smooth manifold locally isomorphic to a Banach space, provided with a second order vector field $F:TM\to TTM$. A standard reference on the subject is the book of Lang \cite[IV.4]{lang}. Recall that such a field verifies $\pi_*\circ F=id_{TM}$, where $\pi:TM\to M$ is the projection map of the tangent bundle, and
$$
F(sv)=(s_M)_* sF(v)\quad\mbox{ for any }s\in \mathbb R,\quad v\in TM. 
$$
Here $s_M:TM\to TM$ denotes the multiplication map $v\mapsto sv$ by $s\in\mathbb R$, and throughout this paper $f_{*}:TX\to TY$ indicates the differential of the smooth map $f:X\to Y$. We use $f_{*x}$ to indicate the differential of $f$ at $x\in X$. 

Let $v\in TM$, and $\beta_v$ the unique integral curve of $F$ with initial condition $v$, that is $\beta_v:I\to TM$, $\beta_v(0)=v$ and 
$$
\frac{d}{dt}\beta_v=F(\beta_v).
$$
Let ${\cal D}_{\exp}\subset TM$ stand for the set of vectors $v$ such that $\beta_v$ is defined at least on the interval $[0,1]$. The exponential map $\exp:{\cal D}_{\exp}\to M$ is defined accordingly to
$$
\exp(v)=\pi (\beta_v(1)),
$$
and the restriction of $\exp$ to each $T_xM$ will be denoted by $\exp_x$. The geodesics of $M$ at $x$ with initial speed $w\in T_xM$ are then given by $\alpha(t)=\pi(\beta_v(t))$, where $v=(x,w)\in TM$. Parallel translation along $\alpha$ will be denoted as
$$
P^t_s(\alpha):T_{\alpha(s)}M\to T_{\alpha(t)}M.
$$

A \textit{tangent norm} on $M$ is a map $b:TM\mapsto \mathbb R^+$ whose restriction to each $T_xM$ is a norm, and it is called a \textit{compatible norm} if the topology induced by $b$ on each $T_xM$ matches the topology induced on it by the Banach space norm.

A \textit{Finsler manifold} is a pair $(M,b)$ of a Banach manifold $M$ and a compatible norm $b$ on $TM$. In this paper we identify $b$ with the subjacent norm $\|\cdot\|_x=b(x)$ of the Banach space, and we measure the length of piecewise smooth curves  $\gamma:[a,b]\to M$ with the usual rectifiable length given by
$$
L_a^b(\gamma)=\int_a^b\|\dot{\gamma}\|_{\gamma}\,dt,
$$
and when $\gamma$ is defined in $I=[0,1]$, we use $L(\gamma)$ for short.

In this paper the term smooth means $C^1$ and with nonzero derivative. The set of piecewise smooth curves in $M$ joining two points $x,y\in M$ will be denoted by $\Omega_{x,y}$, 
$$
\Omega_{x,y}=\{\gamma:[0,1]\to M,\quad\gamma\mbox{ is piecewise smooth},\; \gamma(0)=x,\gamma(1)=y\},
$$
and the distance between points in $M$ is defined as the infimum of the lengths of the piecewise smooth curves joining them, 
$$
d(x,y)=\inf\{L(\gamma),\gamma\in\Omega_{x,y}\}.
$$
Let $Aut(M)=Aut(M,b)$ stand for the group of compatible automorphisms of $M$, which  is the set of diffeomorphisms $\varphi$ of $M$ such that $b\circ \varphi_*=b$. Then the distance defined above is compatible in the sense that the induced topology matches the topology of $M$, and it is invariant for the action of the automorphism group of $M$. See \cite[Prop. 12.22]{upmeier} for a proof of these facts.

A \textit{Finsler manifold with spray} is a Finsler manifold such that the tangent norm $b$ is invariant under parallel transport along geodesics. 

\subsection{Cartan-Hadamard manifolds}\label{ch}

In \cite{neeb} was established by Neeb a definition of semi-negative curvature for Finsler manifolds with spray, we recall it here. A Finsler manifold $M$ with a spray has semi-negative curvature if, for any $x\in M$ and $v\in T_xM\cap {\cal D}_{\exp}$, then
\begin{enumerate}
\item $(\exp_x)_{*v}$ is invertible. 
\item For any $w\in T_xM$,
\begin{equation}\label{semi}
\|(\exp_x)_{*v}(w)\|_{\exp_x(v)}\ge \|w\|_x.
\end{equation}
\end{enumerate}
\medskip

The following Cartan-Hadamard  theorem can be found in \cite[Th. 1.10]{neeb}:

\begin{teo}\label{CH}
Let $M$ be a connected Banach-Finsler manifold with spray with semi-negative curvature. Then $M$ is geodesically complete if and only if $M$ is complete, and in that case for each $x\in M$ the exponential map $\exp_x:T_xM\to M$ is a surjective covering. In particular if $M$ is simply connected $\exp_x$ is an isomorphism for each $x\in M$.
\end{teo}

\begin{rem}\label{mas}
Since $M$ has semi-negative curvature, if $\Gamma$ is any lift (to $T_xM$) of a smooth curve $\gamma\in M$, then
\begin{equation}\label{compare}
L_{T_xM}(\Gamma)\le L_M(\gamma).
\end{equation}
Indeed, since $\exp_x(\Gamma)=\gamma$, then 
$$
\|\dot{ \gamma }\|_{\gamma}= \|(\exp_x)_{* \Gamma}(\dot{\Gamma})\|_{\exp_x(\Gamma)}\ge \|\dot\Gamma\|_{x}.
$$
If $\gamma$ is any smooth curve joining $x$ to $y$ in $M$, let $\Gamma\subset T_xM$ be the unique lift of $\gamma$ such that $\Gamma(0)=0$. Then
$$
L(\gamma)\ge L(\Gamma)\ge \|\Gamma(1)\|_x=L(\gamma_{x,y}),
$$
where $\gamma_{x,y}=\exp_x(t\Gamma(1))$. In particular, given two points $x,y\in M$ there exists a smooth curve $\gamma_{x,y}$ (which is a geodesic) such that $\gamma_{x,y}$ is minimizing for the geodesic distance.
\end{rem}

\begin{rem}\label{otrascortas}
Caution: in spite of the fact that the distance function is convex in a manifold of semi-negative curvature (Theorem \ref{convex} below), there might be other short (i.e. distance minimizing) curves, see Remark \ref{otras} at the Appendix.

However, provided that the norm of $TM$ is \textit{strictly convex},
$$
\|v+w\|=\|v\|+\|w\|\mbox{ implies } v=\lambda w\mbox{ for some }\lambda\in [0,+\infty),
$$
the short curves are unique: Proposition \ref{uni} below proves this fact. Compare to Corollary 6.3 in the book of Lang \cite[Ch. VIII]{lang}, where uniqueness is proved via Gauss Lemma in the Riemann-Hilbert context.
\end{rem}

\begin{defi}
A Cartan-Hadamard manifold is a simply connected, complete Finsler manifold $M$ of semi-negative curvature.
\end{defi}

The question of whether the distance function is convex or not in a Cartan-Hadamard manifold was positively answered in \cite{lim}, let us state this result.

\begin{teo}\label{convex}
Let $M$ be a Cartan-Hadamard manifold, let $\alpha,\beta$ be two geodesics. Then the distance map $f:[0,1]\to\mathbb [0,+\infty)$ given by
$$
f(t)=d(\alpha(t),\beta(t))
$$
is convex.
\end{teo}

\begin{rem}\label{buse}
Let $(X,d)$ be a metric space, which is also a \textit{geodesic length space} in the sense that the distance of $X$ can be computed via the infimum of the length of the rectifiable arcs joining given endpoints in $X$ (see \cite[Section 2.2]{jost}). A geodesic length space is \textit{globally non-positively curved in the sense of Busemann} if for given geodesic arcs $\alpha,\beta$ starting at $x\in X$, the distance map
$$
t\mapsto d(\alpha(t),\beta(t))
$$
is a convex function. Then by the theorem above, any Cartan-Hadamard manifold $(M,d)$, where $d$ is the rectifiable metric given by the Finsler norms, can be regarded as a metric space of nonpositive curvature in the sense of Busemann.
\end{rem}

\section{Metric problems}\label{metricproblems}

Let us begin this section with an elementary inequality (which can be found in the setting of Riemannian manifolds in \cite[Ch. IX, Cor. 3.10]{lang}). It will be useful later, it compares the distance in $M$ with the distance in the tangent linear space. We include a proof for the convenience of the reader. In the context of positive invertible operators (see the Appendix) it is known as the exponential metric increasing property.

\begin{lem}\label{emi}
Let $M$ be a Cartan-Hadamard manifold, let $x\in M$, and $v,w\in T_xM$. Then
$$
\|v-w\|_x\le d(\exp_x(v),\exp_x(w)).
$$
\end{lem}
\begin{proof}
Let $\gamma$ be any piecewise smooth curve in $M$ joining $\exp_x(v)$ to $\exp_x(w)$. Then, by Theorem \ref{CH}, there exists a piecewise smooth curve $\Gamma\subset T_xM$ such that $\gamma=\exp_x(\Gamma)$, with $\Gamma(0)=v$ and $\Gamma(1)=w$. Now, since the differential of the exponential map is an isomorphism,
\begin{eqnarray}
\|w-v\|_x & = & \|\Gamma(0)-\Gamma(1)\|_x=\|\int_0^1 \dot{\Gamma}(t)dt\|_x \nonumber\\
& \le &\int_0^1 \|\dot{\Gamma}(t)\|_x dt=\int_0^1\|(\exp_x)_{*\Gamma}^{-1}(\dot{\gamma})\|_x dt.\nonumber
\end{eqnarray}
The last quantity inside the integral sign is by (\ref{semi}) less or equal than
$$
\|\dot{\gamma}(t)\|_{\exp_x(\Gamma)}=\|\dot{\gamma}(t)\|_{\gamma},
$$
hence $\|w-v\|_x\le L(\gamma)$. Since $\gamma$ is arbitrary, we obtain the asserted inequality.
\end{proof}

\begin{problem}
Equality in the above lemma imposes a rigidity condition; in Theorem \ref{mantiene} we study this problem, in the setting of homogeneous spaces. We would like to know if the following assertions hold in the general setting (here $R(\cdot,\cdot)$ indicates the curvature tensor of $M$ derived from the spray):
\begin{itemize}
\item  $R(v,w)\vrule_{span(v,w)}\equiv 0$ implies that equality holds in Lemma \ref{emi}.
\item If the tangent norms are strictly convex, and equality holds, then $R(v,w)$ restricted to $span(v,w)$ vanishes.
\end{itemize}
This problem is closely related to \cite[Problem 1.2]{neeb}.
\end{problem}

\begin{rem}\label{mil}
Let $x\in M$. Given $v,w\in T_xM$, for $r>0$ let 
$$
s_x(r,v,w)=\frac{r\|v-w\|_x-d(\exp_x(rv),\exp_x(rw))}{r^2 d(\exp_x(v),\exp_x(w))}.
$$
Milnor \cite{milnor} observed that, in the Riemannian setting, sectional curvature can be obtained via the limiting procedure
$$
s_x(v,w)=\frac16\lim_{r\to 0^+} s_x(r,v,w).
$$
Hence this limit (provided it exists) can be used as a suitable definition of curvature. In the present setting, by the inequality in Lemma \ref{emi} above, one has
$$
s_x(r,v,w)\le 0 \quad\mbox{ for any }r>0.
$$
So it seems only natural to ask if the limit exists, and if there are lower bounds. If $M=G/K$ is an homogeneous space, the answer is affirmative, see Section \ref{curvat} below.
\end{rem}

\subsection{Uniform convexity and minimizers}\label{uniconvex}

\begin{defi}
Let $(E,\|\cdot\|)$ be a Banach space. The modulus of convexity of $E$ is the non-negative number
$$
\delta_E(\varepsilon)=\inf\{1- \frac12\|x+y\|: \;\|x\|,\|y\|\le 1, \;\|x-y\|\ge \varepsilon\}.
$$
A Banach space is uniformly convex if $\delta_E(\varepsilon) \;> 0$ for any $\varepsilon\in (0,2]$. A uniformly convex Banach space is strictly convex (cf. Remark \ref{otrascortas}).
\end{defi}

\begin{rem}\label{unibanach}
Assume that $E$ is strictly convex. Then the unique short, piecewise smooth curves of $E$ are the straight segments \cite[Lemma 2.10]{menucci}. That is, if $\gamma$ is a piecewise smooth curve in $E$ joining $0$ to $v$, and $\gamma$ has length $\|v\|$, then $\gamma(t)=tv$. See \cite{menucci} also for examples of infinitely many smooth curves joining given endpoints, in the setting of Banach spaces with a norm that is not strictly convex.
\end{rem}

\begin{prop}\label{uni}
Let $M$ be a Cartan-Hadamard manifold. If the norm of $TM$ is strictly convex, then the geodesics of $M$ are the unique piecewise smooth short paths in $M$. 
\end{prop}
\begin{proof}
Let $\gamma$ be a short curve in $M$, with $\gamma(0)=x$, $\gamma(1)=y$. Let $\Gamma\subset T_xM$ be such that $\exp_x(\Gamma)=\gamma$ and $\Gamma(0)=0$. Then $L(\Gamma)\le L(\gamma)=d(x,y)$ by eqn. (\ref{compare}). Let $v=\Gamma(1)\in T_xM$, let $\alpha(t)=\exp_x(tv)$. Then
$$
d(x,y)\le L(\alpha)=\int_0^1\|\dot{\alpha}\|_{\alpha}\,dt=\|v\|_x
$$
since $\alpha$ is a geodesic, and then $L(\Gamma)\le \|v\|_x$. Since $\Gamma$ joins $0$ to $v$ in $T_xM$, by Remark \ref{unibanach} we obtain that $\Gamma(t)=tv$, or in other words $\gamma(t)=\exp_x(tv)$.
\end{proof}

\begin{defi}
We call $M$ a $p$-uniformly convex Cartan-Hadamard manifold if there exists a positive constant $K_M$ and a number $p\ge 2$ such that 
\begin{equation}\label{mayor}
2(\frac{1}{K_M^p}\|v\|_x^p+\|w\|_x^p)  \le\|v+w\|_x^p+\|v-w\|_x^p, %\le 2^{p-1}\|v\|_x^p+\|w\|_x^p
\end{equation}
for any $x\in M$ and any $v,w\in T_xM$.
\end{defi}

By a result of Ball \textit{et al.} \cite{ball}, a uniformly convex Banach space $E$ has modulus of convexity of power type $p\ge 2$ (that is, $\delta_E(\varepsilon)\ge C\varepsilon^p$) if and only if there exists a constant $K_E>0$ such that a weak Clarkson inequality like (\ref{mayor}) holds. Hence we are assuming that all the tangent spaces of $M$ are of power type $p$, with $K_{T_xM}$ uniformly bounded by $K_M$. This condition guarantees uniform convexity, and in particular, strict convexity of the tangent norms.

This is a convenient generalization of the parallelogram law for the Riemannian metric of Riemann-Hilbert manifolds, since it induces strong convexity result analogous to the Gauss Lemma.  Among the simplest examples of uniformly convex Banach spaces of power type $p$ are the usual $L^p$ measure spaces of functions which were the original concern of Clarkson \cite{clarkson}, and their non-commutative counterpart, the $\bp$ spaces of compact Schatten operators.

\medskip

In this section we prove existence and uniqueness of minimizers in $p$-uniformly convex Cartan-Hadamard manifolds, and give a geometrical characterization of them. 

In what follows, for a given curve $\gamma:I\to M$, let us denote $\gamma(t)=\gamma_t$ for any $t\in I$, then if $\gamma$ is a geodesic, $\gamma_{\frac12}$ is the midpoint between $\gamma_0$ and $\gamma_1$.

\begin{teo}\label{para_p}
Let $M$ be a $p$-uniformly convex Cartan-Hadamard manifold. Let $x,y,z\in M$, and let $\gamma$ be the geodesic joining $y$ to $z$ in $M$. Then
\begin{equation}\label{strongconvex}
\frac{1}{(2K_M)^p}d(y,z)^p \leq \frac{1}{2}(d(x,y)^p+d(x,z)^p)-d(x,\gamma_{1/2})^p.
\end{equation}
\end{teo}
\begin{proof}
Let $a=\gamma_{1/2}\in M$. Note that $d(a,z)=\frac12 L(\gamma)$. Let $v,w\in T_aM$ be such that $y=\exp_a(-v)$, $z=\exp_a(v)$ and $x=\exp_a(w)$. Then by Lemma \ref{emi},
$$
d(x,z)^p=d(\exp_a(w),\exp_a(v))^p\ge \|v-w\|_a^p
$$
and also
$$
d(x,y)^p=d(\exp_a(w),\exp_a(-v))^p\ge \|v+w\|_a^p.
$$
Adding these quantities and using the definition of $p$-uniform convexity, we obtain the stated inequality, since $\|v\|_a=d(a,z)$ and $\|w\|_a=d(a,x)$.
\end{proof}

\begin{rem}
Let $(X,d)$ be a geodesic length space \cite{jost}. Then $X$ is said to be \textit{non-positively curved in the sense of Alexandrov} if for any $x\in X$ and any geodesic segment $\gamma\in X$,
$$
\frac{1}{4}L(\gamma)^2 \leq \frac{1}{2}(d(x,\gamma_0)^2+d(x,\gamma_1)^2)-d(x,\gamma_{1/2})^2.
$$
Nonpositive curvature in the sense of Alexandrov implies nonpositive curvature in the sense of Busemann (see Remark \ref{buse} above for the definition of Busemann nonpositive curvature). 
\end{rem}

\begin{defi}
If $(X,d)$ is a geodesic length space and there exists a positive constant $K$ such that (\ref{strongconvex}) holds for any geodesic $\gamma$ joining $y,z\in X$, we say that \textit{$X$ is a Busemann $p$-space}. 

Let $(X,d)$ be a Busemann $p$-space. A set $C\subset X$ is called convex if, for given $x,y\in K$, the unique geodesic $\gamma_{x,y}$ of $X$ joining $x$ to $y$ is fully contained in $C$.
\end{defi}

Hence the semi-parallelogram law on $M$ (Theorem \ref{para_p}) gives a link with the spaces of nonpositive curvature as studied by Alexandrov, Ballman, Busemann, Gromov \textit{et al.}, see \cite{alexandrov,busemann,jost}. Then Busemann $p$-spaces lie somewhere in between Busemann spaces and Alexandrov spaces, since the metric of the manifold fulfills a strong inequality \textit{a la} Alexandrov, but one does not have the quadratic exponents.

\medskip

\begin{problem}
Evidently Banach-Finsler manifolds $M$ of semi-negative curvature with a $p$-uniformly convex tangent norm are Busemann $p$-spaces, and in that setting the distance between geodesics starting at a common point is a convex function. Is each Busemann $p$-space $(X,d)$ nonpositively curved in the sense of Busemann, for any $p\ge 2$? The proof for $p=2$ \cite[Cor. 2.3.1]{jost} only gives
$$
d(\alpha(t),\beta(t))^p\le t^2 d(\alpha(1),\beta(1))^p+(1-\frac{1}{K^p})\left(L(\alpha)^p+L(\beta)^p\right).
$$
for two geodesics starting at $x\in X$. Even for $K=1$ this is not sufficient.
\end{problem}

We now obtain existence of (unique) minimizers from a convex set to any given point outside it, in the same fashion as in \cite[Ch. 3]{jost}, where it is done for Alexandrov spaces.

\begin{teo}\label{minime}
Let $(X,d)$ be a Busemann $p$-space. Let $C\subset X$ be a convex closed set in $X$ and $x\in X$. Then there exists a unique point $x_C\in C$ such 
that
$$
d(x_C,x)=\min_{y\in C} d(y,x)=d(C,x).
$$
We call $x_C$ the best approximation of $x$ in $C$.
\end{teo}
\begin{proof}
Let $D=d(C,x)$ be the distance between $C$ and $x$. Let $x_n$ be a decreasing minimizing sequence in $C$, that is $\lim\limits_{n\to \infty}d(x_n,x)=D$ and
$$
d(x_n,x)\ge d(x_{n+1},x)
$$
We claim that $\{x_n\}$ is a Cauchy sequence in $C$. Let $\gamma_{n,m}:[0,1]\to M$ be the short geodesic joining $x_n$ to $x_m$ in $M$, which is contained in $C$. Let $m>n$ and let $x_{n,m}\in C$ be the middle point of $\gamma_{n,m}$. Then by the semi-parallelogram law in Theorem \ref{para_p},
$$
\frac12 (d(x_n,x)^p+d(x_m,x)^p)- d(x_{n,m},x)^p\ge \frac{1}{(2K_M)^p}d(x_n,x_m)^p,
$$
and $D \le d(x_{n,m},x)$ since $C$ is convex, hence
$$
\frac12 (d(x_n,x)^p+d(x_m,x)^p)- D^p\ge \frac{1}{(2K_M)^p}d(x_n,x_m)^p,
$$
which proves the claim. To prove uniqueness, assume that $x^1,x^2$ are minimizers in $C$ and let $x^{12}$ be the middle point. If we replace them again in the semi-parallelogram law we obtain
\begin{eqnarray}
0  & = & \frac12 (D^p+D^p)-D^p\ge \frac12 (d(x^1,x)^p+d(x^2,x)^p)- d(x^{12},x)^p \nonumber\\
&\ge& \frac{1}{(2K_M)^p}d(x^1,x^2)^p.\nonumber
\end{eqnarray}
\end{proof}

\bigskip

Let $(X,d)$ be a Busemann $p$-space, let $x_0\in X$ and $\lambda >0$, and let $F:X\to \mathbb R\cup \{\infty\}$ be any function. We define the Moreau-Yoshida approximation $F^{\lambda}$ of $F$ as
$$
F^{\lambda}=\inf\limits_{y\in X} \;\{\lambda  F(y)+d(x_0,y)^p\}.
$$
It is not hard to see using (\ref{strongconvex}) that if $F$ is convex, lower semi-continuous, bounded from below  and not identically $+\infty$, then for every $\lambda>0$ there exists a unique $y_\lambda\in X$ such that
\begin{equation}\label{ylambda}
F^{\lambda}=\lambda\; F(y_{\lambda})+d(x_0,y_{\lambda})^p.
\end{equation}
See \cite[Lem. 3.1.2]{jost} for the details (done there for $p=2$). Then the following result, using our semi-parallelogram laws, has a proof almost identical to that in  \cite[Th. 3.1.1]{jost}, therefore we omit it. 

\begin{teo}
Let $(X,d)$ be a Busemann $p$-space, let $F:X\to \mathbb R\cup \{\infty\}$ be a convex, lower semi-continuous function that is bounded from below and not identically $+\infty$. Let $y_{\lambda}$ be constructed as in (\ref{ylambda}). If $d(x_0,y_{\lambda_n})$ is bounded for some sequence $\lambda_n\to +\infty$, then $\{y_{\lambda}\}_{\lambda>0}$ converges to a minimizer of $F$ as $\lambda\to \infty$.
\end{teo}

\subsubsection{Bruhat-Tits}\label{bruhat}

Existence and uniqueness of minimal balls is guaranteed by the generalized semi-parallelogram laws (Theorem \ref{para_p}), and from there one obtains Bruhat-Tits fixed point theorem and its usual corollaries. The proofs are straightforward and identical to the proofs of the case $p=2$ (see \cite[Section 3]{lang2} for instance), therefore we omit them. The contents of this section are related to \cite[Problem 1.3(2)]{neeb}. In these propositions $M$ is a Busemann $p$-space (in particular, a Cartan-Hadamard manifold with a $p$-uniformly convex tangent norm).

\begin{prop}
Let $S$ be a bounded subset of $M$. Then there exists a unique closed ball $B_r(s_1)\subset M$ of minimal radius $r$ containing $M$. The center $s_1\in S$ is called the circumcenter of $S$.
\end{prop}

\begin{teo}(Bruhat-Tits)
Let $G$ be a group of isometries of $M$. Suppose that $G$ has a bounded orbit (for instance, if $G$ is discrete). Then the orbit of $G$ has a fixed point, for instance the circumcenter.
\end{teo}

\subsection{Metric splittings via convex submanifolds}\label{psplit}

In this section, we give a geometrical characterization of the best approximation $x_C\in C\subset M$, where $C$ is a convex and closed submanifold of a Cartan-Hadamard manifold $M$, and we state a straightforward splitting of $M$ via such submanifolds.

\begin{defi}
Let $X$ be a Banach space, let $S\subset X$ be a linear subspace. The Birkhoff orthogonal $S^{\perp}$ of $S$ is given by
$$
S^{\perp}=\{v\in X:\|v\|\le \|v+s\|\;\mbox{ for any }\, s\in S\}.
$$
\end{defi}

The Birkhoff orthogonal is the analogue of the usual orthogonal in Hilbert spaces. However the Birkhoff orthogonal does not necessarily have a linear structure \cite{james}.

\bigskip

\begin{rem}\label{onlyif}
Let $C\subset M$ be a convex submanifold of a Cartan-Hadamard manifold. Let $z\in C$ and let $x=\exp_z(v)$, with $v$  Birkhoff orthogonal to $T_zC$. Since  (by virtue of the convexity of $C$) for any $y\in C$ we can write $y=\exp_z(s)$, with $s\in T_zC$, then
$$
d(x,z)=\|v\|_z\le \|v-s\|_z\le d(\exp_z(v),\exp_z(s))=d(x,y)
$$
by Lemma \ref{emi}.
\end{rem}

So if $x\in M$ is reached by an orthogonal direction, then it has a closest point in $C$.

\begin{prop}
Let $C\subset M$ be a submanifold of a Cartan-Hadamard manifold $M$. Let $x\in M$ and let $z\in C$. If $z$ is the best approximation of $x$ in $C$, then the initial speed of the geodesic $\alpha$ joining $z$ to $x$ in $M$ is orthogonal to $T_zC$ in the sense of Birkhoff. In addition, if $C$ is a convex submanifold, these conditions are equivalent.
\end{prop}
\begin{proof}
Assume that $d(z,x)\le d(y,x)$ for any $y\in C$, and let us put $v=\dot{\alpha}(0)$, where $\alpha$ is the short geodesic joining $z$ to $x$ in $M$. Let $\tilde{\alpha}(t)=\alpha(1-t)$ be the geodesic joining $x$ to $z$. Then $\tilde{\alpha}(t)=\exp_x(t\exp_x^{-1}(z))$ by the uniqueness of geodesics, hence $\exp_x^{-1}(z)=\dot{\tilde{\alpha}}(0)=-\dot{\alpha}(1)$, and also $P_z^x(\alpha)(v)=\dot{\alpha}(1)$ by parallel translation properties. Hence $P_z^x(\alpha)v=-\exp_x^{-1}(z)$.

Let $A_v:T_zM\to T_zM$ be the linear isomorphism given by $P_x^z (\exp_x^{-1})_{*z}$ (recall that $x=\exp_z(v)$). Then $A_v$ is a contraction by the nonpositive curvature condition (\ref{semi}), and we claim that $A_vv=v$: the curve $\gamma(t)=\exp_x(P_z^x(t-1)v)$ is a geodesic of $M$ with initial data $\gamma(0)=z$ and $\gamma(1)=x$, hence $\gamma(t)=\alpha(t)$ and then
$$
v=\dot{\gamma}(0)=(\exp_x)_{*-P_z^xv}P_z^xv,
$$
hence
$$
P_x^z (\exp_x^{-1})_{*z}v=v.
$$
Let $w\in T_zC$, and let $\beta\subset C$ be any smooth curve such that $\beta(0)=z$, $\dot{\beta(0)}=w$. Consider the convex function $g(t)=d(\beta(t),x)$: if $z$ is the best approximation of $x$ in $C$, then $g'(0^+)\ge 0$. Let $v_t=\exp_x^{-1}(\beta(t))$. Then
\begin{eqnarray}
g(t) &=&\|v_t\|_x=\|P_x^z(\tilde{\alpha})v_t\|_z=\|-v+A_vwt+o(t^2)\|_z \nonumber \\
&\le & \|-v+A_vwt\|_z+o(t^2),
\end{eqnarray} 
since $P_x^z(\tilde{\alpha})v_0=P_x^z(\tilde{\alpha})\exp_x^{-1}(z)=-v$ and $\dot{v}_0=(\exp_x^{-1})_{*z}w$. Now, since $g(0)=\|v\|_z$,   
$$
0\le g'(0^+)\le \|-v+A_vw\|_z-\|v\|_z
$$
by the convexity of the norm, hence $\|v\|_z\le \|-v+A_vw\|_z$ for any $w\in T_zC$. Then, since $v=A_vv$,
$$
\|v\|_z\le \|-A_vv+A_vw\|_z=\|A_v(-v+w)\|_z\le \|v-w\|_z
$$
because $A_v$ is a contraction, and this shows that $v$ is Birkhoff orthogonal to $T_zC$. The last assertion of the proposition follows from Remark \ref{onlyif}.
\end{proof}

\begin{rem}
In \cite{pr}, Porta and Recht prove a splitting theorem for inclusions $N\subset M$ of $C^*$-algebras. In their proof, a key element is the natural linear supplement of the tangent spaces of the submanifold, given by a conditional expectation $E:M\to N$. However, in the setting of $p$-uniformly convex Banach spaces it is natural to replace linear supplements with the Birkhoff orthogonal.
\end{rem}

Since the orthogonal directions in the tangent bundle play a relevant role,  we define \textit{the normal of $C$}
$$
{\mathfrak N}_C=\{(x,v):x\in C,\;v\in T_xC^{\perp}\}\subset TM.
$$ 

We use $\exp:TM\to M$, $\exp(x,v)=\exp_x(v)$ to denote the exponential map of $M$.

\begin{teo}\label{teopconvex}
Let $M$ be a $p$-uniformly convex Cartan-Hadamard manifold, and let $C$ be a convex closed submanifold. Then $\exp:{\mathfrak N}_C\to M$ is a bijection which induces a differentiable structure on ${\mathfrak N}_C$ which makes it diffeomorphic to $M$.
\end{teo}
\begin{proof}
Let $x\in M$, let $z\in C$ be the unique minimizer (Theorem \ref{minime}), with $D=d(x,C)=d(x,z)$. Let $\alpha$ be the unique geodesic in $M$ joining $z$ to $x$. Let $v$ be the initial speed of $\alpha$, then $x=\exp_z(v)$. Note that $\|v\|_z=D$, and also that $v$ is Birkhoff orthogonal to $T_zC$ by the previous proposition, hence $x=\exp(z,v)$ and the map  $\exp$ is surjective. On the other hand, assume that $M\ni x=\exp_y(w)=\exp_z(v)$ with $(z,v),(y,w)\in {\mathfrak N}_C$. Let $D=d(x,C)=\|v\|_z=\|w\|_y$. Then by convexity $d(x,\gamma_{1/2})\ge D$, and by inequality (\ref{strongconvex}),
\begin{eqnarray}
\frac{1}{(2K_M)^p}d(y,z)^p & \le & \frac{1}{2}(d(x,y)^p+d(x,z)^p)-d(x,\gamma_{1/2})^p\nonumber\\
&=&\frac12 D^p+\frac12 D^p-d(x,\gamma_{1/2})^p\le 0,\nonumber
\end{eqnarray}
hence $y=z$ so $\exp$ is injective. With the induced differentiable structure, $\exp$ is a global isomorphism onto $M$, since its differential is everywhere invertible by hypothesis.
\end{proof}

\begin{coro}\label{spliteometrico}
Let $C\subset M$ be a convex closed submanifold of a $p$-uniformly convex Cartan-Hadamard manifold, and let $x\in M$. Then there exists a unique $z\in C$ and $v\in T_zC^{\perp}$ such that $\|v\|_z=d(x,C)$ and $x=\exp_z(v)$.
\end{coro}

\section{Homogeneous spaces}\label{homogeneous}

In this section we assume that $M\simeq G/K$ is an homogeneous reductive space, quotient of Banach-Lie groups. First we recall the basic facts, and include some elementary considerations for the benefit of the reader.

\smallskip

A Banach-Lie group $G$ with an involutive automorphism $\sigma$ is called a \textit{symmetric Lie group} in \cite{neeb}. Let $\g$ be the Banach-Lie algebra of $G$, and let $K=G^{\sigma}=\{g\in G:\sigma(g)=g\}$ be the subgroup of $\sigma$-fixed points. Then the Banach-Lie algebra $\k$ of $K$ is a closed complemented subspace of $\g$; the complement is given by the closed subspace
$$
\p=\{v\in \g : \sigma_{*1}v=-v\},
$$
since the Lie algebra $\k$ matches the set of $\sigma_{*1}$-fixed points. Hence $K$ is a Banach-Lie subgroup of $G$, and the quotient space $M=G/K$ carries the structure of a Banach manifold. We indicate with $q:G\to M$, $g\mapsto gK$ the quotient map and with $Exp:\g\to G$ the exponential map of $G$. We use the short notation $e^v=Exp(v)$ for $v\in \g$ whenever it is possible. Then $q\circ Exp:\p\to M$ is the natural chart around $o=q(1)\in M$ given by the exponential map of $G$, $q\circ Exp=\exp_o\circ q_{*1}$, and a general geodesic of $M=G/K$ is given by
$$
\alpha(t)=ge^{tv}K=q(ge^{tv})
$$
for some $v\in \p$. Note that in particular $M$ is geodesically complete.

Let $h\in G$, let $\mu_h:M\to M$ stand for $\mu_h(q(g))=q(hg)=q(L_hg)$. Then
$$
(\mu_h)_{*q(g)} q_{*g}=q_{*hg}(L_{h})_{*g}.
$$
A generic point in $M$ will be denoted by $q(g)$ for $g\in G$, and we will identify $\p$ with $T_oM$ so a generic vector in $T_{q(g)}M$ will be indicated by $(\mu_g)_{*o} v$ for $v\in \p$. 

We  use $Ad_k$ to denote both the automorphism of $g$ given by $Ad_k(g)=kgk^{-1}$, and also its differential $(Ad_k)_{*1}$ which is an element of ${\cal B}(\g)$, the bounded linear operators acting on $\g$. Note that $\sigma(Ad_ke^{tv})=Ad_k e^{-tv}$ for any $v\in \p$, $k\in K$, so $\sigma_{*1}Ad_k v=-Ad_k v$, hence $\p$ is $Ad_K$-invariant.

\smallskip

\begin{rem}\label{adk}
Since $\sigma$ is a group automorphism, $\sigma_{*1}$ is a Lie-algebra homomorphism, and the relations
$$
[\k,\k]\subset \k,\quad [\k,\p]\subset \p,\quad [\p,\p]\subset \k
$$
follow. In particular, $\p$ is $\ad_{\k}$-invariant as mentioned.

%If $K$ is connected, any element in $K$ can be written as a finite product $k=\prod e^{k_i}$ with $k_i\in \k$. 

The bundle $G\times_K \p$ identifies with $TM$ via $(g,v)\mapsto (q(g),(\mu_g)_{*o}v)$,  the action of $K$ is given by $(g,v)\mapsto (gk^{-1},Ad_kv)$.

Assume that $f=gk$ for some $k\in K$, let $x=q(g)=q(f)$, and assume $(\mu_g)_{*o}v = (\mu_{gk})_{*o}w\in T_xM$. From 
\begin{eqnarray}
(\mu_g)_{*o}v &=& (\mu_{gk})_{*o} w=\frac{d}{dt}|_{t=0} q(gke^{tw})\nonumber\\
&=&\frac{d}{dt}|_{t=0} q(gAd_ke^{tw})=\frac{d}{dt}|_{t=0} q(ge^{tAd_kw})=(\mu_g)_{*o} Ad_k w\nonumber
\end{eqnarray}
we obtain $v=Ad_kw$. These considerations indicate that a natural way to make of $M$ a Finsler manifold is by 
$$
\|(\mu_g)_{*o}v\|_{q(g)}:=\|v\|_{\p}
$$
where $\|\cdot\|_{\p}$ is any $Ad_k$-invariant norm on $\p$. This definition makes parallel translation isometric, since from \cite[p. 135]{neeb} follows that parallel translation along a geodesic $\alpha(t)=q(ge^{tv})$ is given by
$$
P_{0}^{t}(\alpha) =(\mu_{ge^{tv}g^{-1}})_{*q(g)}.
$$
Then the maps $\mu_h:M\to M$ are tautologically isometries since
$$
(\mu_h)_{*q(g)}(\mu_g)_{*o}=(\mu_{hg})_{*o},
$$
and the set $I(G)=\{\mu_g\}_{g\in G}$ is a subgroup of the path-component of the identity of $Aut(M)$ which acts transitively on $M$.
\end{rem}

\begin{rem}\label{expogrupo}
Assume that $G$ is connected. Then $M=G/K$ is a connected and geodesically complete Finsler manifold with spray. Assume that $M$ has semi-negative curvature, and let $\exp:TM\to M$, $(g,v)\to \exp_{q(g)}((\mu_g)_{*o}v)=q(ge^{v})$ stand for the exponential map of $M$, where we identified $TM$ with $G\times_K \p$. In this context, Theorem \ref{CH} says that any element $x\in M$ can be written as $x=q(g e^v)$ for some $v\in \p$.

\begin{rem}\label{expodif}
From now on, whenever it is possible, we shall omit the isomorphism $(\mu_g)_{*o}$ which identifies $\p$ with $T_xM$ when $x=q(g)$, and write $\exp_x(v)=q(ge^v)$ for $x\in M$ and $v\in \p$ when there is no possibility of confusion.
\end{rem}

Let ${\cal B}(\p)$ stand for the bounded linear operators of $(\p,\|\cdot\|_{\p})$. In \cite[Lemma 3.10]{neeb}, the formula for the differential of the exponential map is computed in an homogeneous space. Let $F(z)=z^{-1}\sinh z$, and recall the usual expression for the exponential of the differential map
$$
Exp_{*v}=(L_{e^v})_{*1}\left(\frac{1-e^{-\ad v}}{\ad v} \right)
$$
for $v\in \g$ the Banach-Lie algebra of a Banach-Lie group $G$. Then
$$
(\exp_o)_{*v}=(\mu_{e^v})_{*o}\frac{\sinh\ad v}{\ad v}=(\mu_{e^v})_{*o}F(\ad v)
$$
for any $v\in \p$, since $q_{*e^v}=(\mu_{e^v})_{*o}q_{*1}$, and $q_{*1}$ is essentially the identity on $\p$ and has kernel $\k$. 
\end{rem}

Let us recall some related results for our general framework.

\begin{rem}\label{disi}
If $Z$ is a Banach space, an operator $A\in {\cal B}(Z)$ is called \textit{dissipative} if 
$$
{\mathfrak Re}\; \varphi(Az)\le 0
$$
for some (or equivalently, any) $\varphi\in Z^*$ such that $\varphi(z)=\|z\|$, $\|\varphi\|=1$. This condition is equivalent to the fact that $1-sA$ is expansive and invertible for any $s>0$ \cite{lumer}. 
\end{rem}

What follows is a useful semi-negative curvature criterion for homogeneous spaces, \cite[Prop. 3.15 and Th. 2.2]{neeb}:

\begin{prop}\label{criterio}
Let $M=G/K$ be an homogeneous space with a norm $\|\cdot\|_{\p}:\p\to\mathbb R_{\ge 0}$ which is $Ad_K$-invariant, so $M$ can be regarded as a Finsler manifold. Then the following are equivalent
\begin{enumerate} \item $M$ has semi-negative curvature.
\item For each $v\in \p$, the operator $T_v=-(ad_v)^2|_{\p}$ is dissipative.
\item For each $v\in \p$, the operator $1+(ad_v)^2|_{\p}$ is expansive and invertible.
\item For each $v\in \p$,  $F(\ad v)=\frac{\sinh\ad v}{\ad v}\vrule_{\,\p}$ is expansive and invertible in $\p$.
\end{enumerate}\end{prop}

\begin{rem}\label{cosh}
By mimicking the proof of \cite[Prop. 3.15]{neeb}, it is not hard to see that any entire function $G$, with purely imaginary roots and such that $G(0)=1$ induces by functional calculus a bounded operator $G(\ad v)\in {\cal B}(\p)$, and this operator is invertible and expansive, in particular its inverse is a contraction. We will use this fact repeatedly for $G(z)=\cosh(z)$. See \cite{ineqlaro} for further details on this technique.
\end{rem}

We recall two more results on the fundamental group of $M$ and polar decompositions from \cite[Th. 3.14 and Th. 5.1]{neeb}

\begin{teo}\label{fibras}
Let $(G,\sigma)$ be a connected symmetric Banach-Lie group, $K=G^{\sigma}$ the subgroup of $\sigma$-fixed points. If $M=G/K$ has semi-negative curvature, then 
\begin{enumerate}
\item The exponential map $q\circ Exp:\p\to M$ is a covering of Banach manifolds and 
$$
\Gamma=\{z\in \p: q(e^z)=q(1)\}
$$
is a discrete additive subgroup of $\p \cap Z(\g)$, with $\Gamma\simeq \pi_1(M)$ and $M\simeq \p/\Gamma$. Here $Z(\g)$ denotes the center of the Banach-Lie algebra $\g$. If $v,w\in\p$ and $q(e^v)=q(e^w)$, then  $v-w\in\Gamma$.
\item The polar map $m:\p\times K \to G$, given by $(v,k)\mapsto  e^v k$ is a surjective covering map whose fibers are given by the sets $\{ (v-z,e^z k):\, v\in \p,\, z\in \Gamma, \, k\in K\}$.
\end{enumerate}
\end{teo}

\subsection{Local metric structure and totally geodesic submanifolds}\label{totalgeo}

In what follows we assume that $M=G/K$ is a complete and connected manifold of semi-negative curvature. This whole section is dedicated to the study of the local metric structure of $M$ and the totally geodesic submanifolds of $M$. 

\subsubsection{Local convexity of the geodesic distance}\label{convexity}

First, following \cite{jost}, we prove local convexity results for the geodesic distance (recall that Theorem \ref{convex} was proved in \cite{lim} in the context of simply connected manifolds).

\begin{rem}\label{capa}
Recall that $\Gamma=\exp_o^{-1}\{o\}$ is a discrete additive subgroup of $\p\cap Z(\g)$, since the differential of the exponential map is an isomorphism. Let $\kappa_M\in (0,+\infty)$ stand for the maximum of the positive numbers $r$ such that $0\in \p$ is the unique point of $\Gamma$ in the ball of radius $r$ around it. Note that $\kappa_M=+\infty$ if and only if $M$ is simply connected. 

Note that $\|v-z_0\|_{\p}<\kappa_M/2$ for some $z_0\in \Gamma$ means $\|v-z_0\|_{\p}<\|v-z\|_{\p}$ for any $z\in \Gamma-\{z_0\}$, and this implies that for any $x,y\in M$ and $d(x,y)<\kappa_M/2$, there exists a unique $v\in\p$ such that $\|v\|_{\p}=d(x,y)$ and $y=\exp_x(v)$. Indeed, take any $v'$ such that $\exp_x(v')=y$ and then replace $v'$ with $v=v'-z_0$, where $z_0$ is the element of $\Gamma$ closer to $v'$. 

Moreover, $\alpha(t)=\exp_x(tv)$ is the unique short geodesic joining $x$ to $y$ in $M$, for if $\beta(t)=\exp_x(tw)$ is another geodesic, put $z=v-w\in \Gamma$, and if $z\ne 0$,
$$
d(x,y)=L(\alpha)=\|v\|_{\p}< \|v-z\|_{\p}=\|w\|_{\p}=L(\beta)=d(x,y),
$$
a contradiction.  Note that $\kappa_M$ is diameter of the geodesic balls of $(M,d)$.

With similar argumentation one can show that, for any given $v,w\in \p$, if we put $x=q(e^v)$, $y=q(e^ve^w)$ then $d(x,y)$ is given by $\|w-z_0\|_{\p}$, where $z_0\in \Gamma$ is one of the (possibly many, even infinite) elements of $\Gamma$ which are closer to $w$. Then $\alpha(t)=q(e^ve^{t(w-z_0)})$ is a short geodesic joining $x$ to $y$.
\end{rem}

\begin{prop}\label{loconvex}
Let $x,x'\in M$, let $y=\exp_x(v)$, $y'=\exp_x(v')=\exp_{x'}(w)$, such that  $d(x,y)=\|v\|_{\p}$, $d(x,y')=\|v'\|_{\p}$, $d(x',y')=\|w\|_{\p}$. Let $0<R< \kappa_M/4$.
\begin{enumerate}
\item If $z_0\in \Gamma$ is closer to $v-v'$ than any other $z\in \Gamma$, then
$$
\|v-v'-z_0\|_{\p}\le d(y,y').
$$
In particular, if $y,y'\in B(x,R)$, then 
$$
\|v-v'\|_{\p}\le d(y,y').
$$
\item If $y,y'\in B(x,R)$, then $f:[0,1]\to [0,+\infty)$
$$
f(t)=d(\exp_x(tv),\exp_x(tv'))
$$ 
which gives the distance among the two geodesics starting at $x\in M$, is a convex function.
\item The distance function among the two geodesics joining $x$ to $y$ and $x'$ to $y'$,
$$
g(t)=d(\exp_x(tv),\exp_{x'}(tw))
$$
is also convex, provided that $y,y'\in B(x,R)$ and $d(x',y')<R$.
\item In particular, if $\gamma$ is the short geodesic joining $x',y'$, then $h(t)=d(x,\gamma(t))$ is convex and $\gamma\subset B(x,R)$, provided that $x',y'\in B(x,R)$.
\end{enumerate}
\end{prop}
\begin{proof}
We can assume that $x=o$. Let $\alpha$ be any piecewise curve joining $y$ to $y'$ in $M$. Let $\beta$ be the piecewise smooth lift of $\alpha$ to $\p\simeq T_oM$ such that $\beta(0)=v$. Then there exists $z_{\alpha}\in \Gamma$ such that $\beta(1)=v'-z_{\alpha}$. Hence
$$
\|v-v'-z_0\|_{\p}\le\|v-v'+z_{\alpha}\|_{\p}=\|\beta(1)-\beta(0)\|_{\p}\le L(\beta)\le L(\alpha),
$$
where the last inequality is due to Remark \ref{mas}. This proves the first assertion, since if $y,y'\in B(x,R)$, then $\|v-v'\|_{\p}\le 2R< \kappa_M/2$ and then $z_0=0$.

Let us prove 2. Let $\alpha$ be a short geodesic joining $y$ to $y'$, namely $L(\alpha)=d(y,y')<2R\le \kappa_M/2$. If $\beta \subset T_xM$ is the lift of $\alpha$ such that $\beta(0)=v$, then
$$
\|\beta(1)-v'\|_{\p}\le\| \beta(1)-\beta(0)\|_{\p}+\|v-v'\|_{\p}\le 2d(y,y')<\kappa_M,
$$
hence $\beta(1)=v'$. It will suffice to prove statement $2.$ for $t=1/2$, since $f$ is continuous and a standard argument with the dyadic numbers will complete the proof. 
Let $\overline{\alpha}(t)=q(e^{\beta/2})$. Then certainly $f(1/2)=d(q(e^{v/2}),q(e^{v'/2}))\le L(\overline{\alpha})$ since $\overline{\alpha}$ joins the same endpoints. Note that 
$$
\dot{\overline{\alpha}}=\frac12 F(\ad \beta /2)\dot{\beta},
$$
and on the other hand,
$$
\dot{\alpha}=F(\ad \beta)\dot{\beta}=2F(\ad \beta/2)\cosh(\ad\beta/2)\dot{\beta}.
$$
Hence $\dot{\overline{\alpha}}=\frac12  \cosh(\ad \beta/2)^{-1}\dot{\alpha}$. By Remark \ref{cosh}, $\|\dot{\overline{\alpha}}\|_{\overline{\alpha}}\le \frac12 \|\dot{\alpha}\|_{\alpha}$, hence $L(\overline{\alpha})\le \frac 12 L(\alpha)=\frac12 d(y,y')=\frac12 f(1)$, which proves $2.$

To prove $3$, note that $g(t)\le f(t)+ f'(t)$, where $f$ is the function of item $2.$ and $f'$ is the corresponding function for the geodesics starting at $y'$ and ending at $x,x'$ respectively. Then $f,f'$ are convex functions and 
$$
g(1/2)\le \frac12 ( f(1)+ f'(1)) =\frac12 (g(1)+g(0) ).
$$

The last statement follows choosing $y=x$, and then 
$$
h(t)=d(x,\gamma(t))\le td(x,x')+(1-t)d(x,y')<R.
$$
\end{proof}

\subsubsection{A formula for the geodesic distance}\label{distanceformula}

We will use $\log: G\cap U\to \g$ to denote the inverse function of the exponential map of $G$ (restricted to a suitable neighborhood $U$ of $1\in G$ to obtain a diffeomorphism).

\medskip

Since $d(\exp_x(rv),\exp_x(rw))=d(o,q(e^{-rv}e^{rw}))$ for any $x\in M$ and $v,w\in \p$, for small $r\in\mathbb R$ we have
$$
d(\exp_x(rv),\exp_x(rw))=\frac12\|\log(e^{-rv}e^{2rw}e^{-rv})\|_{\p}.
$$
Indeed, if $\gamma(r)$ is a continuous lift of $q(e^{-rv}e^{rw})$ to $\p$ with $\gamma(0)=0$, then $\|\gamma(r)\|_{\p}=d(o,q(e^{-rv}e^{rw}))$ and on the other hand
$$
e^{2\gamma(r)}=e^{-rv}e^{2rw} e^{-rv}.
$$
So if $r$ is small enough in order to ensure that the exponential is a local diffeomorphism, then 
$$
2\gamma(r)=\log(e^{-rv}e^{2rw}e^{-rv}).
$$
\begin{coro}
Let $x\in M$ and $v,w\in \p$. Let 
$$
{\cal R}(v,w)=\frac{1}{12}[v+w,[w,v]]=\frac{1}{12} \left[\ad^2_w(v)-\ad^2_v(w)\right].
$$
Then for small $r\in\mathbb R$,
\begin{eqnarray}
d(\exp_x(rv),\exp_x(rw)) & = & \frac12\|\log(e^{-rv}e^{2rw}e^{-rv})\|_{\p}=\frac12\|\log(e^{-rw}e^{2rv}e^{-rw})\|_{\p}\nonumber \\
&= &\|r(w-v)+r^3 {\cal R}(v,w)+o(r^4)\|_{\p},\nonumber
\end{eqnarray}
where $\log$ denotes the analytic inverse of the exponential map of $G$, defined in a suitable neighborhood of $1\in G$.
\end{coro}
\begin{proof}
The first two equalities follow from the previous discussion. Iterating the Baker-Campbell-Hausdorff formula, one obtains
\begin{eqnarray}
d(\exp_x(rv),\exp_x(rw)) &= &\frac12\|2r(w-v)+r^3 \frac{2}{12}[v+w,[w,v]]+o(r^4)\|_{\p}\nonumber\\
&= &\|r(w-v)+r^3 \frac{1}{12}[v+w,[w,v]]+o(r^4)\|_{\p},\nonumber
\end{eqnarray}
which holds for $r$ small enough.
\end{proof}

\subsubsection{Sectional curvature}\label{curvat}

With the tools of the previous section we now return to the subject matter of Remark \ref{mil}.

\begin{prop}
Let $x\in M$, $v,w\in \p$. Let $r>0$ and
$$
s_x(r,v,w)=\frac{r\|v-w\|_{\p}-d(\exp_x(rv),\exp_x(rw))}{r^2 d(\exp_x(v),\exp_x(w))}.
$$
Then $s_x(v,w)=\lim\limits_{r\to 0^+} s_x(r,v,w)$ exists and
$$
0\ge s_x(v,w)\ge 1-\frac{\|v-w+{\cal R}(v,w)\|_{\p}}{\|v-w\|_{\p}}\ge -\frac{\|{\cal R}(v,w)\|_{\p}}{\|v-w\|_{\p}}.
$$
In particular if ${\cal R}(v,w)=0$ then $s_x(v,w)=0$ for any $x\in M$.
\end{prop}
\begin{proof}
Note first that by the previous corollary,
$$
\lim\limits_{r\to 0^+} \frac1r d(\exp_x(rv),\exp_x(rw))=\|w-v\|_{\p}.
$$
Since a norm is a convex function, then
$$
\lim\limits_{r\to 0^+} \frac{1}{r^2}\left( \|w-v\|_{\p} -\|w-v+r^2 {\cal R}(v,w)+o(r^2)\|_{\p}\right)
$$
exists and its is fact equal to $-J_{v-w}({\cal R}(v,w))$, that is, (minus) the subdifferential of the norm at the point $v-w$, computed in the direction of ${\cal R}(v,w)$. Moreover, 
$$
\|x\|_{\p}-\|x-y\|_{\p}\le J_x(y)\le \|x+y\|_{\p}-\|x\|_{\p}.
$$
See for instance \cite[Prop. 4.1]{aubin}. Then 
$$
\lim\limits_{r\to 0^+} \frac{1}{r^2} \|w-v\|_{\p} -\|w-v+r^2 {\cal R}(v,w)+o(r^2)\|_{\p}\ge \|v-w\|_{\p}-\|w-v+{\cal R}(v,w)\|_{\p},
$$
thus $s_x(v,w)=\lim\limits_{r\to 0^+} s_x(r,v,w)$ exists, is nonpositive, and by the computation above
$$
s_x(v,w)\ge 1-\frac{\|v-w+{\cal R}(v,w)\|_{\p}}{\|v-w\|_{\p}}.
$$
The right-hand inequality stated in the proposition follows straight from the triangle inequality.
\end{proof}

\subsubsection{On the distortion of the metric}\label{distortion}

We now assume for convenience that $M\simeq G/K$ is simply connected. In our present setting, if we choose $x=o$, our concern now is the inequality stated as
\begin{equation}\label{emigrupo}
\|v-w\|_{\p}\le d(q(e^v),q(e^w)),
\end{equation}
where $v,w\in\p$. We have seen that it implies that sectional curvature in $G/K$ is nonpositive. If $v,w\in \p$ commute, the exponential of the linear span of $v,w$ is a $2$-dimensional flat in $M$, and clearly equality holds in equation (\ref{emigrupo}); this condition $[v,w]=0$ is equivalent (by Jacobi's theorem) to the commutativity of the local flows of the Jacobi fields $V,W$ (induced by $v,w$ respectively). In the infinite dimensional setting, one obtains a weaker notion made explicit in the following theorems. The definitions and considerations of Remark \ref{unibanach} will be used here.

\begin{prop}\label{stric}
Let $v,w\in \p$. If $M=G/K$ is a Cartan-Hadamard manifold and the norm $\|\cdot\|_{\p}$ is strictly convex, then 
$$
\|v-w\|_{\p}= d(q(e^v),q(e^w))
$$
implies $\ad_v^2(w)=\ad_w^2(v)=0$.
\end{prop}
\begin{proof}
Let $\alpha$ be the short geodesic of $M$ joining $q(e^v)$ with $q(e^w)$, 
$$
\alpha(t)=q(e^{v}e^{tz})\;\mbox{and}\; q(e^{v}e^{z})=q(e^{w}),
$$
where $z$ is the unique lift to $\p$ of $q(e^{-v}e^{w})$; note that $\|z\|_{\p}=d(q(e^v),q(e^w))=\|v-w\|_{\p}$. Let $\gamma$ be the unique lift to $\p$ of $\alpha$, $\gamma(0)=v$ and $\gamma(1)=w$; by Remark \ref{mas}
$$
L(\gamma)\le L(\alpha)=\|v-w\|_{\p}.
$$
Since the norm of $\p$ is strictly convex, it must be $\gamma(t)=(1-t)v+tw$, so
$$
q(e^{(1-t)v+tw})=q(e^{v}e^{tz}).
$$
Differentiating at $t=0$ we obtain
$$
(\mu_{e^{v}})_{*o}q_{*1}\frac{1-e^{-\ad v}}{\ad v}(w-v)=(\mu_{e^{v}})_{*o}q_{*1}z
$$
by Remark \ref{expodif}, that is
$$
F(\ad v)(w-v)=z.
$$
Then $\|F(\ad v)(w-v)\|_{\p}=\|z\|_{\p}=\|w-v\|_{\p}$. If $\varphi\in \p^*$ is the unique norming functional of $w-v$, since $-\ad^2_v$ is dissipative by Proposition \ref{criterio},
\begin{eqnarray}
2\|w-v\|_{p} & = & 2\varphi(w-v)\le \varphi( 2(w-v) +\frac{1}{\pi^2}\ad^2_v(w-v)) \nonumber\\
& \le & \| 2(w-v) +\frac{1}{\pi^2}\ad^2_v(w-v)\|_{\p},\nonumber
\end{eqnarray}
that is
$$
2\|w-v\|_{\p}\le \|w-v+(1+\frac{1}{\pi^2}\ad^2_v)(w-v)\|_{\p}.
$$
On the other hand 
$$
\|w-v\|_{\p}\le \|(1 +\frac{1}{\pi^2}\ad^2_v)(w-v)\|_{\p}\le \|F(\ad v)(w-v)\|_{\p}=\|w-v\|_{\p}
$$
since $F(z)=\prod_{n\ge 1} \left(1+\frac{z^2}{n^2\pi^2} \right)=\left(1+\frac{z^2}{\pi^2} \right)\prod_{n\ge 2} \left(1+\frac{z^2}{n^2\pi^2} \right)$, and each factor is an expansive operator, thus
$$
\|w-v\|_{\p}=\|(1 +\frac{1}{\pi^2}\ad^2_v)(w-v)\|_{\p}.
$$
Then
$$
\|w-v+ (1 +\frac{1}{\pi^2}\ad^2_v)(w-v)\|_{\p}=\|w-v\|_{\p}+\|(1 +\frac{1}{\pi^2}\ad^2_v)(w-v)\|_{\p},
$$
and since the norm is strictly convex and both elements have the same norm, it must be 
$$
w-v=(1 +\frac{1}{\pi^2}\ad^2_v)(w-v)=w-v+\frac{1}{\pi^2}\ad^2_v(w-v).
$$
Interchanging $w,v$ gives $\ad^2_w(v)=0$ also.
\end{proof}

\begin{teo}\label{mantiene}
Let $v,w\in \p$. Let $M=G/K$ be a Cartan-Hadamard manifold, and let
\begin{enumerate}
\item $[v,[v,w]]=[w,[v,w]]=0$.
\item $\|v-w\|_{\p}=d(q(e^v),q(e^w))$.
\end{enumerate}
Then $1.$ implies $2.$, and if the norm of $M$ is strictly convex, $2.$ is equivalent to $1.$
\end{teo}
\begin{proof}
The previous proposition gives $2\Rightarrow 1$. On the other hand, if $[v,[v,w]]=[w,[v,w]]=0$, then by the Baker-Campbell-Hausdorff formula,
$$
e^{-v}e^{w}=e^{w-v-\frac12 [v,w]}=e^{w-v}e^{-\frac12 [v,w]}
$$
since higher order commutators vanish. Thus $q(e^{-v}e^w)=q( e^{w-v})$, and if $\alpha(t)=q(e^ve^{t(w-v)})$, then $\alpha$ is the unique geodesic joining $q(e^v)$ to $q(e^w)$ in $M$, hence $d(q(e^v),q(e^w))=\|w-v\|_{\p}$.
\end{proof}

\begin{rem}
In the finite dimensional setting, if $[v,[v,w]]=[w,[v,w]]$, and $B:\g\times\g$ denotes the Killing form of $\g$ (i.e $B(x;y)=Tr(\ad x \;\ad y)$ where $Tr$ denotes the usual trace of ${\cal B}(\g)$), then
$$
B([v,w];[v,w])=B(v;[w,[v,w])=B(v;[v,[v,w]])=B([v,w];[v,v])=0.
$$
So if $\g$ is semi-simple, the condition $[v,[v,w]]=[w,[v,w]]$ implies $[v,w]=0$. From Proposition \ref{stric} follows that such condition is guaranteed if 
$$
\|v-w\|_{\p}=d(q(e^v),q(e^w)),
$$
so in this setting the (apparently weaker) metric condition is equivalent to the commutativity of local flows, and then to the presence of a $2$-dimensional flat. This line of reasoning can be extended to the infinite dimensional setting in the presence of a trace (Hilbert-Schmidt operators or $L^*$-algebras), see \cite{cocoeste} for full details.
\end{rem}

\begin{problem}
Find necessary and sufficient conditions on the norm of $\p$ in order to ensure that if $v,w\in\p$ and $[v,[v,w]]=0$, then $[v,w]=0$.
\end{problem}

\subsubsection{Totally geodesic submanifolds}\label{totgeo}

Some of the results in the following proposition can be originally found in \cite{mostow}, in the setting of the group of positive invertible $n\times n$ matrices. They express the standard relation between totally geodesic submanifolds and Lie triple systems. In the finite dimensional (Riemannian) setting, the standard reference would be the book of Helgason \cite{helga}. In \cite{pr4}, the authors study exponential sets in $C^*$-algebras with similar techniques and recently, the results in \cite{mostow} were extended to Hilbert-Schmidt operators \cite{larro}. 

\begin{prop}\label{triples}(Exponential sets)
Let $M=G/K$ be a connected manifold of semi-negative curvature, where $T_oM\simeq \p$. Let $\s\subset \p$ be a closed linear space and let $C=q(e^{\s})$. Then the following conditions are equivalent, and we call $C$ an exponential set.
\begin{enumerate}
\item $[[v,w],s]\in \s$ for any $v,w,s\in\s$.
\item $\ad^2_s(\s)\subset \s$ for any $s\in \s$.
\item $F(\ad v)=\frac{\sinh \ad v}{\ad v}\in {\cal B}(\p)$ is an isomorphism of $\s$ for any $v\in\s$.
\item If $v,w\in \s$, and $\beta\subset T_oM\simeq\p$ is a lift of $\alpha(t)=q(e^ve^{tw})$ such that $\beta(0)\in \s$, then $\beta\subset \s$. 
\end{enumerate}
\end{prop}
\begin{proof}
Let $v,w,s\in \s$. Then 
$$
[[v,w],s]=-ad^2_{v-w}(s)+ad^2_v(s)+ad^2_w(s)
$$
by the Jacobi identity. This shows that $2.$ is equivalent to $1.$

Assume that $2.$ holds, then certainly $3.$ holds since the series expansion of $F(z)=z^{-1}\sinh(z)$ has only even powers of $z$. If $3.$ holds, replacing $v$ with $tv$ yields 
$$
\s\ni s_t=F(\ad tv)w=w+\frac{1}{6}t^2\ad_v^2w+o(t^4),
$$
hence $\frac16 \ad_v^2w=\lim_{t\to 0} \frac{s_t-w}{t^2}\in \s$.

Assume that $2.$ holds, and let $v,w\in\s$. Consider the flow $F_{v,w}:\p\to \p$ given by 
$$
F_{v,w}(z)=\frac{\ad z}{\sinh(2\ad z)}\cosh \ad v (w).
$$
Then $F_{v,w}$ is a Lipschitz map, and if $2.$ holds, $F_{v,w}(\s)\subset\s$. We claim that if $\beta(t)\in\p$ is the smooth lift of $q(e^ve^{tw})$ with $\beta(0)=v$, then  $\dot{\beta}=F_{v,w}(\beta)$, and this will prove that $\beta\subset\s$ by the uniqueness of the solution of the differential equation $\dot{x}=F_{v,w}(x)$ in the Banach space $(\s,\|\cdot\|_{\p})$. To prove the claim $\dot{\beta}=F_{v,w}(\beta)$, write $e^{\gamma}=e^ve^{tw}k$ for some $k(t)\in K$. The derivative of $q(e^{\beta})$ gives
$$
(\mu_{e^{\beta}})_{*o}q_{*1}\frac{1-e^{-\ad\beta}}{\ad\beta}\dot{\beta},
$$
and the derivative of $q(e^ve^{tw})$ gives
$$
(\mu_{e^ve^{tw}})_{*o} q_{*1}w=(\mu_{e^ve^{tw}})_{*o}w=(\mu_{e^{\beta}})_{*o}(Ad_{k^{-1}}w).
$$
Then 
$$
q_{*1}\frac{1-e^{-\ad\beta}}{\ad\beta}\dot{\beta}=Ad_{k^{-1}}w$$
or, since $1-e^x=1-\cosh(x)+\sinh(x)$ and $q_{*1}(\k)=\{0\}$ (and $q_{*1}$ is the identity on $\p$),
$$
\frac{\sinh(\ad\beta)}{\ad\beta}\dot{\beta}=e^{-\beta}e^vwe^{-v}e^{\beta}=e^{-\ad\beta}e^{\ad v}w.
$$
Multiplying by $e^{\ad\beta}$ we obtain $\frac{e^{2\ad\beta}-1}{\ad\beta}\dot{\beta}=e^{\ad v}w$, and applying $q_{*1}$ at both sides, $\frac{\sinh(2\ad\beta)}{\ad\beta}\dot{\beta}=\cosh(\ad v)w$, 
showing that $\dot{\beta}=F_{v,w}(\beta)$.

Assume that $4.$ holds, and let $\gamma_s\subset\s$ be as above, $q(e^{\gamma_s})=q(e^{sv}e^{tw})$. Then by the computation above, with $t\to 0$, we obtain
$$
\s\ni \dot{\gamma_s}(0)=\frac{\ad sv}{\sinh\ad sv}w=w-\frac{4}{3}s^2\ad_v^2w+o(s^4).
$$
Then $-\frac43\ad_v^2w=\lim\limits_{s\to 0}\frac{\dot{\gamma_s}(0)-w}{s^2}\in \s$, showing that $2.$ holds.
\end{proof}

\begin{coro}\label{exponentialset}
Let $C=q(e^{\s})$ be an exponential set in $M$, let $V\in\p$ be an open ball of radius strictly less than $\kappa_M/2$. Then:
\begin{enumerate} 
\item The charts $(V\cap \s,\exp_x\vrule_{V\cap \s})$, for $x\in C$, give an atlas of $C$ which makes of $C$ an immersed  differentiable manifold $C\subset M$, with a topology which is possibly finer than the topology of $M$.
\item $T_xC=(\mu_{e^s})_{*o}\s$ for any $x=q(e^s)\in C$. In particular $\exp_x(T_xC)=C$ for any $x\in C$, i.e. $C$ is totally geodesic in $M$.
\end{enumerate}
\end{coro}
\begin{proof}
For the first statement note that $\exp_x(\s)\subset C$ by Proposition \ref{triples}, and note that $\exp_x\vrule_V$ gives an isomorphism $\exp_x\vrule_V:V\to \exp_x(V)\subset M$ by Remark \ref{capa}. Then the proposed charts are bijective, and moreover the transition maps give isomorphisms between open neighborhoods of $\s$ since the exponential of $M$ is a local isomorphism and $\s$ is a closed linear subspace of $\p$ which (by Proposition \ref{triples}) is stable for the action of the differential of the exponential map at $x=q(e^v)$, given by $F(\ad v)=\frac{\sinh\ad v}{\ad v}$ by Remark \ref{expogrupo}. Then $C$ with the topology and differentiable structure induced by the atlas is an immersed submanifold since $\s\subset \p$ is closed.

The second assertion is elementary, and its proof follows combining $1.$ with Proposition \ref{triples}.
\end{proof}

\begin{rem}
If $G_C\subset G$ is a connected, involutive Banach-Lie group, with Banach-Lie algebra $\g_C\subset \g$, then $\sigma$ allows us to write $\g_C=\p_C\oplus \k_C$, where $\p_C=\p\cap \g_C$ and $\k_C=\k\cap \g_C$. Then $q(G_C)=q(e^{\p_C})\subset M$ is a totally geodesic immersed submanifold.
\end{rem}

\begin{defi}
Let $[\s,\s]$ stand for the closure of the linear span of the elements $[v,w]\in \g$, where $v,w\in \s$. Then $\s\cap [\s,\s]=\{0\}$ since $\s\subset \p$ and $[\s,\s]\subset \k$. Let us agree to call a Banach-Lie algebra $\g_C\subset \g$ \textit{involutive} if $\sigma_{*1}\g_C=\g_C$, and a connected Banach-Lie group $G_C\subset G$ involutive if $\sigma(G_C)=G_C$, or equivalently, if its Lie algebra is involutive. Let $p\in {\cal B}(\p)$ be an idempotent, $p^2=p$. Let $\s=Ran(p)$, $\s'=Ran(1-p)$, so $\p=\s\oplus\s'$. In this case, we say that $\s$ is \textit{split} in $\p$. We say that $C=q(e^{\s})$ is a reductive submanifold if $C$ is totally geodesic and, in addition, $\ad^2_{\s}(\s')\subset \s'$.
\end{defi}

See Remark \ref{clasica} for a brief discussion on these definitions in the classical (Riemannian, finite dimensional) setting, see also item 6. in the following proposition.

\begin{prop}\label{grupejo}
Let $M=G/K$ be a connected manifold, with semi-negative curvature. Let $\s\subset \p$ be a closed linear space. Assume that $\ad^2_{\s}(\s)\subset\s$ and let $\g_{\s}=\s\oplus [s,s]$. 
\begin{enumerate}
\item $\g_{\s}$ is an involutive Banach-Lie algebra and it can be enlarged to a connected involutive Banach-Lie group $G_{\s}\hookrightarrow G$. 
\item Let $K_{\s}=K\cap G_{\s}$. If $C=q(e^{\s})$, then  $G_{\s}/K_{\s}\simeq C$, and $C$ is a totally geodesic, immersed submanifold of $M$.
\item The group $G_{\s}$ acts isometrically and transitively on $C$. 
\item $M$-parallel transport along geodesics in $C$ preserves tangent vectors of $C$.
\item $C$ is a split submanifold if and only if $\s$ is split in $\p$.
\item Let $\k_{\s}=[\s,\s]$, and let $K_C\hookrightarrow G_{\s}$ stand for the Banach-Lie group generated by $\k_{\s}$. Then $C$ is reductive if and only if $Ad_{K_C}$ is a group of isometries of both $\s$ and $\s'$.
\item If $C$ is an embedded submanifold of $M$, then $K_{\s}$ is a Banach-Lie subgroup of $G_{\s}$, $K_C$ is the connected component of the identity of $K_{\s}$, and $G_{\s}/K_{\s}\simeq C$ as homogeneous spaces.
\end{enumerate}
\end{prop}
\begin{proof}
That $\g_C$ is a Lie algebra follows from the Jacobi identity. Since it is a subalgebra of $\g$, which is the Banach-Lie algebra of the Banach-Lie group $G$, it can be integrated as claimed \cite{swie}, and this settles $1.$

To prove $2$, note that if $g\in G_{\s}$ then $g=\prod e^{s_i} e^{k_i}$, where $s_i\in \s$ and $k_i\in [\s,\s]$. Then $q(g)=q(\prod e^{s_i'})$, where $s_i'\in \s$ since
$$
e^{k_i}e^{s_{i+1}}e^{k_{i+1}}=e^{Ad_{e^{k_i}}(s_{i+1})}e^{k_i}e^{k_{i+1}},
$$
and on the other hand $Ad_{ e^{[v,w]}  }  s=e^{  ad_{[v,w]}  }s\in \s$ if $v,w\in\s$ by Proposition \ref{triples}. Then there exists $s\in\s$ such that $q(g)=q(e^s)\in C$ by Proposition \ref{triples}. Then $q\,\vrule_{G_{\s}}$ gives the isomorphism of $G_{\s}/K_{\s}$ with $C$. That $C$ is a totally geodesic immersed submanifold follows from Corollary \ref{exponentialset}. 

To prove 3, note that the transitive and isometric action of $G_{\s}$ is given by the maps $\mu_g$, with $g\in G_{\s}$: if $v\in \s$, then $\mu_g(q(e^v))=q(ge^v)=q(\prod e^{s_i} e^{k_i} e^v)=q( e^{s'_i}e^{v'})$ by the argument above, where $s_i', v'\in\s$, and then $\mu_g(q(e^v))\in C$ by Proposition \ref{triples}. 

To prove 4, recall (Remark \ref{adk}) that $M$-parallel transport along $\alpha(t)=q(e^se^{tv})$ is given by 
$$
(\mu_{e^se^ve^{-s}})_{*q(e^s)}.
$$
Then if $s,v\in \s$, parallel transport along $\alpha$ from $\alpha(0)=q(e^s)$ to $\alpha(1)=q(e^se^v)$ of a vector $(\mu_{e^s})_{*o}w\in T_xC$ gives $(\mu_{e^se^v})_{*o}w$. By Proposition \ref{triples}, there exists $l\in \s$ and $k\in K$ such that $e^l=e^se^v$, and then
$$
P_0^1(\alpha)(\mu_{e^s})_{*o}w =(\mu_{e^l})_{*o} Ad_k  w.
$$
But $Ad_k w= e^{-\ad l}e^{\ad s}e^{\ad v}w \in \p\cap \g_{\s}$, hence $Ad_k w\in \s$, which proves that $P_0^1(\alpha)$ maps $T_{\alpha(0)}C$ to $T_{\alpha(1)}C$.

Item 5 is obvious: $C$ is a split submanifold if and only if $\s$ is slit in $\p$. 

To prove 6, note that each $k\in K_C$ can be written as a finite product $k=\prod e^{l_i}$, with $l_i\in [\s,\s]$. Then $C$ is reductive if and only if $\s$ and $\s'$ are $\ad_{[\s,\s]}$-invariant. 

Finally, if $C$ is an embedded submanifold of $M$, then $q^{\s}=q\vrule_{\s}$ gives the topological identification $G_{\s}/K_{\s}=C$, and inspection of the action of $q^{\s}_{*1}$ shows that $K_{\s}$ is a Banach-Lie subgroup of $G_{\s}$ with Banach-Lie algebra $[\s,\s]$.

\end{proof}

\begin{prop}\label{cuasi}(Locally convex sets)
Let $C=q(e^{\s})$ be an exponential set in $M$. Then the following statements are equivalent, and we call $C$ a locally convex set.
\begin{enumerate}
\item There exists $0<\varepsilon<\kappa_M/2$ such that if $x,y\in C$ and $d(x,y)<\varepsilon$, then if $\alpha(t)=q(e^ve^{tz})$ is the unique short geodesic of $M$ joining $x$ to $y$, then $z\in \s$ and moreover $\alpha\subset C$.
\item There exists $0<\delta<\kappa_M/2$ such that $d(\Gamma-\Gamma\cap\s,\s)\ge \delta$.
\item There exists $0<R<\kappa_M/2$ such that if $U=\{v\in \p: \|v\|_{\p}<R\}$, then $\exp_x(U)\cap C=\exp_x(U\cap \s)$ for any $x\in C$.
\end{enumerate}
\end{prop}
\begin{proof}
Assume that $2.$ does not hold. Then, given $0<\varepsilon<\kappa_M/2$, there exists $z_0\in \Gamma-\s$ such that $d(z_0,\s)<\varepsilon/2$. Take $s\in \s$ such that $\|s-z_0\|_{\p}\le\varepsilon$. Let $w=s-z_0\notin\s$, $x=o$, $y=q(e^w)=q(e^s)\in C$. Then $d(x,y)=\|w\|_{\p}=\varepsilon$ by Remark \ref{capa}, so $\alpha(t)=q(e^{tw})$ is the unique short geodesic of $M$ joining $x$ to $y$. But $\alpha$ does not have initial speed in $\s$, so $1.$ does not hold. 

Now assume that $2.$ holds for some $0<\delta<\kappa_M/2$, and let $x=q(e^s)\in C$. Take $R=\delta$, and note that the inclusion $\exp_x(U\cap \s)\subset \exp_x(U)\cap C$ always holds due to Proposition \ref{triples}. Let $v\in U$, and assume that $q(e^se^v)\in C$, namely $q(e^se^v)=q(e^w)$ with $w\in \s$. Then there exists $s'\in\s$ (again due to Proposition \ref{triples}) such that $q(e^v)=q(e^{-s}e^w)=q(e^{s'})$. Then there exists $z\in\Gamma $ such that $s'-v=z$. If $z\in \s$, we are done since $q(e^se^v)=q(e^se^{s'-z})\in \exp_x(U\cap \s)$. If $z\notin \s$, then $\delta \le \|s'-z\|_{\p}=\|v\|_{\p}<R=\delta$ which is absurd, so $z\in \s$. This shows that $2.$ implies $3.$

Assume that $3.$ holds for some $R>0$, and let $\varepsilon =R$. Let $x=q(e^v), y=q(e^w)\in C$ with $d(x,y)<\varepsilon$, let $\alpha(t)=q(e^ve^{tz})$ be the unique short geodesic of $M$ joining $x$ to $y$, namely $\|z\|_{\p}=d(x,y)$ and $q(e^ve^z)=q(e^w)$. Then, due to $3.$, there exists $s\in U\cap \s$ such that $q(e^ve^z)=q(e^ve^l)$, hence there exists $z_0\in \Gamma $ such that $z-l=z_0$. Since $\|z_0\|_{\p}\le \|z\|_{\p}+\|l\|_{p}<2R$, then $z_0=0$ and $z=l\in \s$. That $\alpha\subset C$ follows from Proposition \ref{triples}, so we have shown that $3.$ implies $1.$
\end{proof}

\begin{coro}\label{cerrada}
Let $C=q(e^{\s})$ be a locally convex set in $M$, let $U\subset\p$ be an open ball around $0$ of radius $R$, where $R$ is as in the previous proposition. Then:
\begin{enumerate} 
\item The set $C$ is an embedded submanifold of $M$, $\exp_x\vrule_{U\cap \s}:U\cap \s\to C\cap \exp_x(U)$ is a topological isomorphism when $C$ is given the subspace topology. It is also a diffeomorphism which gives an atlas which makes of $M$ and immersed embedded submanifold of $C$.
\item With the induced spray and metric, $C$ is a Banach-Finsler manifold with spray of semi-negative curvature, with exponential map $\exp_x^C=\exp_x\vrule_{\s
}$ given by restriction. The fundamental group of $C$ is given by $\Gamma_{\s}=\Gamma\cap \s$, and $C\subset M$ is a closed metric subspace. 
\item If $K_{\s}=K\cap G_{\s}$, then $K_{\s}$ is a Banach-Lie subgroup of $G_{\s}$, and $C\simeq G_{\s}/K_{\s}$ as homogeneous spaces.
\end{enumerate}
\end{coro}
\begin{proof}
That $C$ is an embedded submanifold follows from the fact that if $V\subset U$ is  open in $\p$, then $\exp_x(V)\cap C=\exp_x(V\cap \s)$, because $\exp_x(V)\subset \exp_x(U)$ and then (put $x=q(e^v)$ with $v\in \s$), $q(e^ve^z)\in C$ for $z\in V$ implies $q(e^z)=q(e^s)$ for some $s\in U\cap \s$, so $z=s$ since $z,s\in U$ and $2R<\kappa_M$. 

That $\exp_x\vrule_{U\cap \s}:U\cap \s\to C\cap \exp_x(U)$ is a diffeomorphism follows from Proposition \ref{cuasi}.

The second assertion follows from the fact that the norm of $C$ is compatible since $C$ and $M$ share the topology, and the exponential map of $C$ is just the restriction of the exponential map of $M$, and then at each point its differential is an invertible expansive operator. Then Theorem \ref{fibras} applies.

Now we prove that $C\subset M$ is a closed subspace. If $x_n\to x$ with $x_n\in C$, take $n_0$ such as $d(x_n,x)<R/2$ for any $n\ge n_0$. Let $x_{n_0}=q(e^{v_{n_0}})$, and consider $z_n=\mu_{x_{n_0}}^{-1}x_n$, $z=\mu_{x_{n_0}}^{-1}x$. Since $d(x_n,x_{n_0})<R$, there exists $v_n\in \s\cap U$ such that $x_n=q(e^{v_{n_0}}e^{v_n})$ and $\|v_n\|_{\p}=d(z_n,o)<R$. Then $z_n=q(e^{v_n})\in C$, $d(z_n,z)\to 0$ and then $d(z_n,z_m)<R$. Hence $\|v_n-v_m\|_{\p}\le d(z_n,z_m)$ by Proposition \ref{loconvex}. Since $\s$ is complete, there exists $v_0\in \s$ such that $v_n\to v_0$. Let $z_0=q(e^{v_0})\in C$. Then $d(z,z_0)\le d(z,z_n)+d(z_n,z_0)=d(x,x_n)+d(q(e^{v_n}),q(e^{v_0}))$, hence $z=z_0\in C$, so $x=\mu_{x_{n_0}}(z)\in C$.

The last assertion follows from Proposition \ref{grupejo}, since $C$ is an embedded submanifold.
\end{proof}

\begin{prop}(Convex sets)
Let $C=q(e^{\s})$ be a locally convex set in $M$. Then the following statements are equivalent, and we call $C$ a convex set:
\begin{enumerate}
\item $C$ is geodesically convex: if $x,y\in C$, any geodesic of $M$ joining $x$ to $y$ is entirely contained in $C$.
\item $\Gamma$ is an additive subgroup of $\s$.
\item For any $x\in C$, $\exp_x(v)\in C$ implies $v\in \s$. In particular $\exp_x\vrule_{\s}$ is a global chart of $C$, and $C$ is an immersed embedded submanifold of $M$.
\end{enumerate}
\end{prop}
\begin{proof}
Assume first that $C$ is convex, and let $z\in \Gamma$. Then $\alpha(t)=q(e^{tz})$ joins $o$ to $o$, hence $\alpha\subset C$. In particular, since $T_oC=\s$ by the previous corollary, $\dot{\alpha}(0)=z\in \s$, so $\Gamma\subset \s$.

Assume now that $\Gamma\in \s$, let $x=q(e^s)\in C$, and let $v\in \p$ such that $q(e^se^v)\in C$, namely there exists $w\in \s$ such that $q(e^se^v)=q(e^w)$. Then by Proposition \ref{triples} there exists $s'\in \s$ such that $q(e^v)=q(e^{-s}e^w)=q(e^{s'})$. Since $v-s'\in \Gamma\subset \s$, then $v\in \s$. 

Let $x,y\in C$, let $\alpha(t)=q(e^ve^{tz})$ be a geodesic of $M$ joining $x=q(e^v)$ to $y$. If $3.$ holds, then at $t=1$ we obtain $z\in \s$ and then $\alpha\subset C$ by Proposition \ref{triples}.
\end{proof}

\begin{coro}
Let $C=q(e^{\s})$ be a convex submanifold in $M$. Then if $v,w\in \s$ and $\beta\subset T_oM\simeq \p$ is any lift of $\alpha(t)=q(e^ve^{tw})$, then $\beta\subset\s$.
\end{coro}
\begin{proof}
Let $\beta\in \p$ be any lift of $\alpha(t)=q(e^ve^{tw})$. If $v,w\in \s$, then $\alpha\subset C$ by Proposition \ref{triples}, and moreover $q(e^{\beta(0)})=q(e^v)\in C$. If $C$ is convex, then $3.$ holds in the above proposition, and if we put $x=o$ we obtain $\beta(0)\in \s$, and then $\beta\subset \s$ by Proposition \ref{triples}.
\end{proof}

\subsection{Splitting theorems for expansive submanifolds}\label{splitheo}

In this section, we prove straightforward generalizations of the results due to Corach, Porta and Recht in \cite{cpr1,pr,pr4} for $C^*$-algebras, so we would like to refer to these splitting results as CPR splittings.

\medskip

In what follows, we assume  that $M=G/K$ is connected and complete, of semi-negative curvature. We also assume that $C=q(e^{\s})$ is a locally convex reductive submanifold of $M$.

\begin{defi}\label{reflectiva}
If $C=q(e^{\s})$ is a locally convex reductive submanifold, and in addition $\|p\|=1$ we say that $C$ is an \textit{expansive reductive} submanifold of $M$.
\end{defi}

\begin{rem}
Let $p\in {\cal B}(\p)$ be an idempotent with $\|p\|=1$. Then $\p=\s\oplus\s'$, where $\s=Ran(p)$, $\s'=Ker(p)$, and
$$
\|s\|_{\p}=\|p(s+s')\|_{\p}\le \|s+s'\|_{\p}
$$
for any $s\in\s$, $s'\in\s'$. This shows that $\|p\|=1$ if and only if $\s$ is a subset of the Birkhoff orthogonal of $\s'$, and there is a Banach space isometric isomorphism $\p/\s'\simeq \s$ when $\p/\s'$ is given the quotient norm. Moreover, it easy to check that the following are equivalent:
\begin{itemize}
\item $\|p\|=1$
\item $\s$ is the Birkhoff orthogonal of $\s'$.
\item $1-p=Q_{\s}$, where $Q_{\s}$ indicates the metric projection to $\s$.
\end{itemize}
Obviously the same assertions hold is we replace $p$ with $1-p$ and $\s$ with $\s'$. We call vectors in $\s'$ \textit{normal directions}.
\end{rem}

\begin{lem}\label{laproyeslipsitz}
Let $0<R\le \frac{\kappa_M}{8}$, let $x_0\in C$. Let $x,y\in B(x_0,R)\cap C$, and let $v,w\in \s'$ such that $\|v\|_{\p},\|w\|_{\p}<R$. Let $f:[0,+\infty)\to [0,+\infty)$ be the distance among the two normal geodesics,
$$
f(t)=d(\exp_x(tv),\exp_y(tw)).
$$
Then if $C$ is expansive, $f$ is increasing. If $f$ is increasing for any such $x,y\in C$, $v,w\in \s'$, then $C$ is expansive.
\end{lem}
\begin{proof}
We assume as always that $x_0=o$. Put $x=q(e^r)$, $y=q(e^s)$, with $\|s\|_{\p},\|r\|_{\p}<R$. Then $f(t)=d(q(e^re^{tv}),q(e^se^{tw}))$ is a convex function by Proposition \ref{loconvex}, and it is increasing if and only if $f'(0^+)=\lim_{t\to 0^+}\frac{f(t)-f(0)}{t}\ge 0$. Let $l_0\in \s$ be such that $q(e^{l_0})=q(e^{-r}e^s)$, and $\|l_0\|_{\p}=d(x,y)$ (such element exists by Proposition \ref{cuasi}). Let $k\in K$ be such that $e^{l_0}k=e^{-r}e^s$, let $\beta(t)=q(e^{-tv}e^{-r}e^se^{tw})=q(e^{-tv}e^{l_0}e^{tw'})$,  where $w'=Ad_kw\in \s'$. Note that $d(o,\beta(t))=f(t)\le t\|v\|_{\p}+R+t\|w\|_{\p}< \kappa_M/2$. Then if we put $l_t\in  \p$ the smooth lift of $\beta(t)$ to the ball $B(0,\kappa_M/2)$ in $\p$, we have  $q(e^{l_t})=q(e^{-tv}e^{l_0}e^{tw'})$, and $\|l_t\|_{\p}=d(o,\beta(t))=f(t)$ since $\|l_t\|_{\p}< \kappa_M/2$.

Let $\varphi_0\in\p^*$ be a linear functional such that $\|\varphi_0\|=1$, $\varphi_0(l_0)=\|l_0\|_{\p}=d(x,y)$, and let $\varphi=\varphi_0\circ p$. Then  $\varphi(\s')=\{0\}$. Let $g(t)=\varphi(l_t)$. Note that $g(0)=\varphi(l_0)=f(0)$. If $C$ is expansive, then $\varphi(l_t)\le f(t)$. Then
$$
\frac{f(t)-f(0)}{t}\ge \frac{g(t)-g(0)}{t},
$$
for $t>0$, and we  will show that $g'(0)=0$ to prove that $f$ is increasing. From $q(e^{l_t})=q(e^{-tv}e^{l_0}e^{tw'})$ we obtain
$$
\frac{\sinh\ad l_0}{\ad l_0}\dot{l}_0=q_{*1} (-e^{-\ad l_0}v+w')=w'-\cosh(\ad l_0)v,
$$
hence
$$
\dot{l}_0=F^{-1}(\ad l_0)w'-H(\ad l_0)v,
$$
with $F(z)=z^{-1}\sinh(z)$ and $H(z)=z\coth(z)$, which are both series in $z^2$. Then $g'(0)=\varphi(\dot{l}_0)=\sum\alpha_k \varphi(\ad_{l_0}^{2k} w')-\sum\beta_k\varphi(\ad_{l_0}^{2k} v)=0$ since $\ad_{l_0}^2(\s')\subset \s'$.

Assume now that $f$ is increasing for $x=o$, $v=0$, and for given $l_0\in \s, w_0\in \s'$, put $y=\exp_x(l_0)\in C$. Assume first that $\|l_0\|_{\p},\|w_0\|_{\p}<R$. Put $w=Ad_{k^{-1}}w_0$. Then, in the notation of the first part of the proof, $w'=w_0$ and
$$
f(t)=\|l_t\|_{\p}=\|l_0+t\dot{l}_0+o(t^2)\|_{\p}\le \|l_0+t\dot{l}_0\|_{\p}+o(t^2),
$$
and if $f$ is increasing
$$
0\le f'(0^+)\le \lim\limits_{t\to 0^+}\frac{\|l_0+t\dot{l}_0\|_{\p}-\|l_0\|_{\p}}{t}\le \|l_0+\dot{l}_0\|_{\p}-\|l_0\|_{\p}
$$
by the convexity of the norm. By the computation above, $\dot{l}_0=F^{-1}(\ad l_0)w_0$. Then
$$
\|l_0\|_{\p}\le \|l_0+F^{-1}(\ad l_0)w_0\|_{\p}=\|F^{-1}(\ad l_0)(l_0+w_0)\|_{\p}\le \|l_0+w_0\|_{\p},
$$
since $F^{-1}$ is a contraction. If now $l\in \s$ and $w\in \s'$, replacing them with a convenient positive multiple we obtain that $\|l\|_{\p}\le \|l+w\|_{\p}$, and this shows that $\|p\|=1$.
\end{proof}

\begin{lem}
The sets 
$$
\s_{R}\oplus \s'_R=\{v\in \p: v=s+s', s\in\s,s'\in \s', \|s\|_{\p},\|s'\|_{\p}<R\}
$$
are open neighborhoods of $0\in \p$.
\end{lem}
\begin{proof}
Let $s+s'\in \s_{R}\oplus \s'_R$ with $\|s\|_{\p}=R-\delta$, $\|s'\|_{\p}=R-\delta'$, and let $\varepsilon=\min \{\delta/\|p\|,\frac{\delta'}{1+\|p\|}\}$. We claim that $B(s+s',\varepsilon)\subset \s_{R}\oplus \s'_R$. Let $t+t'\in B(s+s',\varepsilon)$, then
$$
\|t\|_{\p}\le \|t-s\|_{\p}+\|s\|_{\p}\le \|p\|\|t-s+t'-s'\|_{\p}+R-\delta<\|p\|\varepsilon +R-\delta<R,
$$
and on the other hand
\begin{eqnarray}
\|t'\|_{\p}&\le & \|t'-s'\|_{\p}+R-\delta'\le \|t'-s'+t-s\|_{\p}+\|t-s\|_{\p}+R-\delta' \nonumber\\
&<& \varepsilon+\|p\|\varepsilon+R-\delta'<R.\nonumber
\end{eqnarray}
\end{proof}

\begin{lem}\label{zerosection}
Let $x_0=q(e^{s_0})\in C$,  $R>0$ and 
$$
\Omega_{x_0}^{R}=\{\exp_y(v), \, y\in C,\, d(x_0,y)<R, \,v\in \s', \|v\|_{\p}<R\}.
$$
Let $E_{x_0}:\p\to M$ be given by $E_{x_0}(s+s')=q(e^{s_0}e^se^{s'})=\exp_y((\mu_g)_{*o}s')$, where $y=q(e^{s_0}e^s)\in C$ and $g=e^{s_0}e^s$. Then there exists $\varepsilon>0$  (and strictly smaller than $\kappa_M/8$) such that $E_{x_0}:\s_{\varepsilon}\oplus \s'_{\varepsilon}\to \Omega_{x_0}^{\varepsilon}$ is a diffeomorphism, and in particular $\Omega_{x_0}^{\varepsilon}\subset M$ is open. The set
$$
NC^{\varepsilon}=\{\exp_y(v):y\in C,\, v\in\s',\,\|v\|_{\p}<\varepsilon\}
$$
is an open neighborhood of $C$ in $M$.
\end{lem}
\begin{proof}
Let $\alpha(t)=t(s+s')$ with $s+s'\in\p$. Then $E_{x_0}\circ \alpha(t)=q(e^{s_0}e^{ts}e^{ts'})$, hence $(E_{x_0})_{*0}(s+s')=s+s'$, so by the inverse function theorem there exists an open neighborhood $U$ of $0\in \p$ and an open neighborhood $V$ of $x_0\in M$ such that $E_{x_0}$ restricted to them is a diffeomorphism. Shrinking, we can assume that $U=\s_{\varepsilon}\oplus \s'_{\varepsilon}$ and then $\Omega_{x_0}^{\varepsilon}=E_{x_0}(U)$. The last statement is due to the fact that $NC^{\varepsilon}=\cup_{x_0\in C}\Omega_{x_0}^{\varepsilon}$.
\end{proof}

\begin{rem}\label{proye}
Assume that $C$ is locally convex, reductive and expansive. Let $\exp_y(v)=\exp_{y'}(v')\in \Omega_{x_0}^{\varepsilon}$, with $y,y'\in B(x_0,R)$ and $v,v'\in \s'_{\varepsilon}$. If  $\varepsilon<\kappa_M/8$, then by Lemma \ref{laproyeslipsitz},  $y=y'$. Moreover $v-v'\in \Gamma$ but since $\|v-v'\|_{\p}\le 2\varepsilon$, then $v=v'$.

Let $\pi_{x_0}:\Omega_{x_0}^{\varepsilon}\to C\cap B(x_0,\varepsilon)$ be the local projection to $C$, $\pi_{x_0}(\exp_y(v))=y$. Then by Lemma \ref{zerosection},  if $\alpha\subset \Omega_{x_0}^{\varepsilon}$ is the short geodesic starting at $z=q(e^{s_0}e^se^{s'})$ with initial speed $w\in\p$, $\|w\|_p=L(\alpha)$, then $\alpha(t)=q(e^{s_0}e^{s_t}e^{v_t})$ for some smooth curves $s_t\in \s$, $v_t\in \s'$ with $s_0=s$ and $v_0=s'$. Hence 
$\pi_{x_0}\circ \alpha(t)=q(e^{s_t})$, and 
$$
(\pi_{x_0})_{*z}w=F(\ad s)\dot{s}_0
$$
follows. Since $\pi_{x_0}$ is a contraction by Lemma \ref{laproyeslipsitz},
$$
d(q(e^{s_0}e^s),q(e^{s_0}e^{s_t}))=d(\pi_{x_0}(z),\pi_{x_0}(\alpha(t))\le d(z,\alpha(t))=L_0^t(\alpha)=t\|w\|_{\p}.
$$
If $\gamma$ is a smooth curve in $\s$ such that $\gamma(0)=0$, $q(e^{\gamma})=q(e^{-s}e^{s_t})$, and $\|\gamma\|_{\p}=d(q(e^{-s}),q(e^{s_t}))$, from $t\|w\|_{\p}\ge \|\gamma(t)\|_{\p}$ follows $\|F(\ad s)\dot{s}_0\|_{\p}\le \|w\|_{\p}$, or equivalently,  $\|(\pi_{x_0})_{*z}\|\le 1$.
\end{rem}

\begin{teo}
Let $x_0\in C$, let $\varepsilon$ be as in Lemma \ref{zerosection} and put
$$
\Omega_{x_0}=\{\exp_y(v): y\in C,\, d(y,x_0)<\varepsilon,\, v\in \s'\}.
$$
Let $k\in\mathbb N_0$, let $\eta_{k}:\Omega_{x_0}^{\varepsilon}\to \Omega_{x_0}$ be $\eta_k(\exp_y(v))=\exp_y(2^k v)$. Then the differential of $\eta_k$ is an expansive invertible operator. In particular, $\eta_k$ is a local isomorphism.
\end{teo}
\begin{proof}
We assume as always that $x_0=o$. Let $z=\exp_y(v)=q(e^{s}e^v)\in \Omega_{x_0}^{\varepsilon}$, and let $\alpha(t)=q(e^se^{v}e^{tw})$ for $s+v\in \s_{\varepsilon}\oplus \s'_{\varepsilon}$. Then for $t$ small enough, $\alpha(t)\in \Omega_{x_0}^{\varepsilon}$, so we consider $\beta=\eta\circ \alpha$ to compute $\eta_{*z}w=\dot{\beta}(0)$.
Let $s_t+v_t\in \s_{\varepsilon}\oplus \s'_{\varepsilon}$ be such that $\alpha(t)=q(e^{s_t}e^{v_t})$, with $s_0=s$ and $v_0=v$. Then a straightforward but tedious computation yields
$$
w= \frac{\sinh\ad v}{\ad v}\dot{v}_0 +\cosh(\ad v)\frac{\sinh\ad s}{\ad s}\dot{s}_0-\sinh(\ad v)\left(\frac{1-\cosh\ad s}{\ad s}\right)\dot{s}_0.
$$
Replacing $v_t$ with $2^k v_t$, yields
$$
(\eta_k)_{*z}w= \frac{\sinh 2^k\ad v}{\ad v}\dot{v}_0 +\cosh(2^k\ad v)\frac{\sinh\ad s}{\ad s}\dot{s}_0-\sinh(2^k\ad v)\left(\frac{1-\cosh\ad s}{\ad s}\right)\dot{s}_0.
$$
Using the identities $\sinh(2z)=2\sinh(z)\cosh(z)$, $\sinh^2(z)+1=\cosh^2(z)$ and $\cosh(2z)=\cosh^2(z)+\sinh^2(z)$, we obtain
$$
(\eta_k)_{*z}w=2\cosh(2^{k-1}\ad v)(\eta_{k-1})_{*z}w-F(\ad s)\dot{s}_0.
$$
From the previous remark, the last term matches with $(\pi_{x_0})_{*z}w$. Now by Remark \ref{cosh}, $\cosh(2^{k-1}\ad v)$ is an expansive invertible operator of $\p$, hence
$$
(\eta_k)_{*z}w=\cosh(2^{k-1}\ad v)\left[2(\eta_{k-1})_{*z}-\cosh^{-1}(2^{k-1}\ad v)(\pi_{x_0})_{*z}\right]w.
$$
The proof is on induction on $k$. If $k=0$, there is nothing to prove since $\eta_0=id$. Assume then that $(\eta_{k-1})_{*z}$ is expansive and invertible for any $z\in \Omega_{x_0}^{\varepsilon}$. Then
$$
(\eta_k)_{*z}w=\cosh(2^{k-1}\ad v)(\eta_{k-1})_{*z}\left[2-(\eta_{k-1})_{*z}^{-1}\cosh^{-1}(2^{k-1}\ad v)(\pi_{x_0})_{*z}\right]w.
$$
If $u\in\p$ and $\varphi\in \p^*$ is any unit norming functional for $u$, then if we put $A_k=(\eta_{k-1})_{*z}^{-1}\cosh^{-1}(2^{k-1}\ad v)(\pi_{x_0})_{*z}$,
\begin{eqnarray}
\varphi\left(A_k u-u\right) &\le & \|A_k u\|_{\p}-\varphi(u)\nonumber\\
&\le &\|u\|_{\p}-\|u\|_{\p}=0,\nonumber
\end{eqnarray}
which shows that $A_k-1$ is a dissipative operator on $\p$, and by Remark \ref{disi}, the operator
$$
1-(A_k-1)=2-(\eta_{k-1})_{*z}^{-1}\cosh^{-1}(2^{k-1}\ad v)(\pi_{x_0})_{*z}
$$
is expansive and invertible in $\p$. Then $(\eta_k)_{*z}$ is also expansive and invertible. In particular $\eta_k$ is a local isomorphism by the inverse function theorem.
\end{proof}

\begin{teo}
Let $C=q(e^{\s})$ be a locally convex expansive reductive submanifold in $M$. Then $\Omega_{x_0}$ is an open neighborhood of $\exp_{x_0}(\s')$ in $M$ and
$$
NC=\{\exp_x(v):x\in C,\, v\in \s'\}
$$
is an open neighborhood of $C$ in $M$. The inequality $d(\eta_k x,\eta_k y)\ge d(x,y)$ holds for $x,y\in \Omega_{x_0}$ sufficiently close.
\end{teo}
\begin{proof}
Since  $\Omega_{x_0}=\bigcup\limits_{k\in\mathbb N_0} \eta_k\Omega_{x_0}^{\varepsilon}$, then $\Omega_{x_0}$ is an open set in $M$. Clearly $NC$ contains $C$, and on the other hand $NC$ is the union of open sets $NC=\bigcup\limits_{x_0\in C} \Omega_{x_0}$.

If $\alpha$ is a short geodesic joining $\eta_k x$ to $\eta_k y$, and $x,y$ are close enough, then $\alpha\subset\Omega_{x_0}$ and $\beta=\eta_k^{-1}\circ\alpha$ is a smooth curve in $\Omega_{x_0}^{\varepsilon}$ (for some $\varepsilon>0$) joining $x$ to $y$. Then
$$
d(x,y)\le L(\beta)\le L(\alpha)=d(\eta_k x,\eta_k y).
$$
\end{proof}

\begin{teo}\label{spliteo} (CPR splittings for Cartan-Hadamard manifolds) Let $C=q(e^{\s})$ be a reductive expansive submanifold in $M$, and assume that $M$ is simply connected. Then if $v,w\in \s'$ and $x,y\in C$, the distance function  $f:[0,+\infty)\to [0,+\infty)$
$$
f(t)=d(\exp_x(tv),\exp_y(tw))
$$
is increasing. For each $k\in \mathbb N_0$, the map $\eta_k: NC^{\varepsilon}\to NC$ given by $\eta_k\exp_x(v)=\exp_x(2^k v)$ is injective, and it is an isomorphism onto its image, with expansive differential. Moreover $NC=M$, namely
$$
M=\{\exp_x(v):x\in C,\,v\in \s'\},
$$
so for any $v\in  \p$ there exists a unique $\s\in \s$ and a unique $s'\in \s'$ such that $q(e^v)=q(e^se^{s'})$. The projection map $\pi:M\to C$ is contractive for the geodesic distance.
\end{teo}
\begin{proof}
If $M$ is simply connected, $C$ is a closed, convex, embedded immersed submanifold of $M$ by Corollary \ref{cerrada}. In Lemma \ref{laproyeslipsitz}, we can take $R=+\infty$ since $\kappa_M=+\infty$. This proves the first assertion, and moreover, it shows that $\eta_k$ is injective. Then $NC=\cup_{k\in\mathbb N_0}\eta^k NC^{\varepsilon}\subset M$ is an open set in $M$, and moreover $\pi:NC\to M$ is contractive by Remark \ref{proye}, since $\pi(\exp_y(v))=\pi(\exp_y(\lambda v))$ for real $\lambda$, and then the argument in that remark applies. To finish, we claim that $NC$ is closed in $M$: for consider $x_n\in NC$ such that $x_n\to x\in M$. Then any $x_n$ can be uniquely written as $x_n=\exp_{y_n}(v_n)$, with $y_n\in C$ and $v_n\in \s'$. Since $\pi$ is a contraction, $y_n$ is a Cauchy sequence in $C$, and since $C$ is closed in $M$, there exists $y_0\in C$ such that $\lim y_n=y_0$. Then, by Lemma \ref{emi}, 
$$
\|v_n-v_0\|_{\p}\le d(q(e^{v_n}),q(e^{v_0}))= d(\exp_{y_n}(v_n),\exp_{y_n}(v_0))=d(x_n,\exp_{y_n}(v_0)).
$$
Letting $n\to \infty$ gives $v_n\to v_0\in \s'$, and then $x=\lim x_n=\lim \exp_{y_n}(v_n)=\exp_{y_0}(v_0)\in NC$.
\end{proof}

\begin{problem}
Extend the results of Theorem \ref{spliteo} to arbitrary Cartan-Hadamard manifolds (i.e. the general setting of Section \ref{uniconvex}).
\end{problem}

The relationship between this last result and Theorem \ref{teopconvex} of Section \ref{psplit} is presented below:

\begin{teo}\label{split}
Let $C=q(e^{\s})\subset M$ be an expansive reductive submanifold, let $z \in NC$, $z=\exp_x(v)$ for some $x\in C$, $v\in \s'$, and assume that $\|v\|_{\p}=d(x,z)\le \kappa_M/8$. Then $x$ is (locally) the best approximation to $z$ in $C$ (and then $d(z,C\cap B(z,\kappa_M/8))=\|v\|_{\p}$) if and only if $\|1-p\|=1$.
\end{teo}
\begin{proof}
Assume first that $\|1-p\|=1$. Since the action of $G_{\s}$ is transitive and isometric on $C$, we can assume that $x=o$, hence $z=q(e^v)$. Let $y=q(e^r)\in C$, with $r\in\s$ such that $\|r\|_{\p}=d(x,y)\le \kappa_M/8$. Then $d(z,y)\le \kappa_M/4$ and
$$
d(x,z)=\|v\|_{\p}=\|(1-p)(v-r)\|_{\p}\le \|v-r\|_{\p}\le d(q(e^v),q(e^r))=d(x,y),
$$
where the last inequality follows from Proposition \ref{loconvex}.  

On the other hand, if $d(x,o)\le d(x,y)$ for any $y\in C\cap B(x,\kappa_M/8)$, consider the function $f(t)=d(q(e^v),q(e^{ts}))$, with $f:[0,+\infty)\to (0,+\infty)$ and $s\in \s$ with $\|s\|_{\p}\le \kappa_M/8$.  Then the claim implies that $f$ has a local minimum at $t=0$. In particular, $f'(0^+)\ge 0$. As in the proof of Lemma \ref{laproyeslipsitz}, $f(t)=\|\gamma(t)\|_{\p}$, where $q(e^{\gamma})=q(e^{-v}e^{ts})$, with $\gamma(0)=-v$ and $\dot{\gamma}(0)=G(\ad v)s$. Hence
$$
0\le  f'(0^+)\le \|-v+G(\ad v)s\|_{\p}-\|v\|_{\p},
$$
and then $\|v\|_{\p}\le \|-v+s\|_{\p}$ for any $s\in \s$ small enough, so replacing $v,s$ with convenient multiples, we obtain $\|1-p\|=1$.
\end{proof}

\begin{rem}\label{clasica}
In the setting of finite dimensional (Riemannian) symmetric spaces $M=G/K$, a \textit{symmetric} submanifold $C\subset M$ is a submanifold such that there exists an involutive isometry $\varepsilon_0$ of $M$ such that $\varepsilon_0(K)=K$, $\varepsilon_0(C)=C$, and $(\varepsilon_0)_*(v)=(-1)^jv$, with $j=0$ if $v\in T_oC^{\perp}$ and $j=-1$ if $v\in T_oC$. In this context, it is easy to see that a submanifold is symmetric if the supplement $\s'$ of its tangent space $\s$ at $o=q(1)$ is a Lie triple system, $\ad_{\s'}(\s')\subset \s'$. A submanifold $C\subset M$ is called \textit{reflective} if it is both totally geodesic and symmetric. In the Riemannian setting, if $\s'=\s^{\perp}$, one also has the dual relations
$$
\ad_{\s'}^2(\s)\subset \s, \quad \ad_{\s}^2(\s')\subset \s'
$$ 
due to the fact that $\ad_v^2$ is self-adjoint for any $v\in \p$. Hence any reflective submanifold is reductive.

\medskip

In our infinite dimensional setting, it is natural to consider, given a Cartan-Hadamard manifold $M=G/K$, a second involutive automorphism $\tau$ of $G$ which commutes with $\sigma$. Let
$$
\u_+=\{v\in \g: \tau_{*1}v=v\},\quad \u_-=\{v\in \g:\tau_{*1}v=-v\}.
$$
Then if we put $\s=\p\cap \u_-$ and $\s'=\p\cap \u_+$, the conditions 
\begin{equation}\label{adre}
\ad_{\s}^2(\s)\subset \s, \quad \ad_{\s}^2(\s')\subset \s'
\end{equation}
are automatically fulfilled, so $C=q(e^{\s})$ is a reductive submanifold according to our definition \ref{reflectiva}.

If we define $\tau_0:M\to M$ as the involution given by $\tau_0(q(g))=q(\tau^{-1}(g))$, then if $M$ is simply connected, we can compute $\tau_0(q(e^v))=q(e^{-\tau_{*1} v})$ for any $v\in \p$, and $C$ is the set of $\tau_0$-fixed points. If $\tau_0$ is an isometry of $M$, since this is equivalent to the fact that $\tau_{*1}\vrule_{\,\p}$ is an isometry of $\p$, the reductive submanifold $C\subset M$ is expansive according to our definition \ref{reflectiva}, due to the fact that the projection $p$ onto $\s$ is given by $p=(1-\tau_{*1}\vrule_{\,\p})/2$. Moreover, since $1-p=(1+\tau_{*1}\vrule_{\,\p})/2$, the normal bundle gives the best approximation from $C$. Hence isometric involutions $\tau$ which commute with $\sigma$ induce reductive submanifolds for which Theorems \ref{spliteo} and \ref{split} apply, inducing a metric splitting as in Corollary \ref{spliteometrico} of Section \ref{psplit}.
\end{rem}

\begin{rem}
If $C=q(e^{\s})$ is a reflective submanifold, in the sense that 
$$
\ad^2_{\s}(\s)\subset\s,\quad  \ad^2_{\s}(\s')\subset\s',\quad  \ad^2_{\s'}(\s)\subset\s,\quad \ad^2_{\s'}(\s')\subset\s',
$$
then one obtains that $NC=\{\exp_x(v):\,x\in C,\;v\in \s'\}$ is open in $M$ with a more direct proof. One has to observe that if $v=s+s'\in\p$, and $w=t+t'$ (here $s,t\in\s$ and $t,t'\in \s'$ as usual), then the map $E:\p\to M$, of Lemma \ref{zerosection}, $E(v)=q(e^se^{s'})$, has its differential in form of a block matrix relative to $\s\oplus\s'$ given by
$$
E_{*v}w=\left(\begin{array}{cc} 
\cosh(\ad s')\frac{\sinh(\ad s)}{\ad s} & 0 \\
\\
\sinh(\ad s')\frac{(\cosh(\ad s)-1)}{\ad s} & \frac{\sinh(\ad s')}{\ad s'}
\end{array}\right)
\left(\begin{array}{c}t\\ 
\\
t' \end{array}\right).
$$
Then $E$ is a local isomorphism at any $v\in\p$, so $E(\p)$ is open in $M$.
\end{rem}

\subsubsection{CPR splittings for Banach-Lie groups}\label{groups}

Let $(G,\sigma)$ be an involutive Banach-Lie group. Let $\tau=\sigma_{*1}$, $\g=\p\oplus\k$ be the $\tau$-decomposition of $\g$. Assume that the Banach-Lie algebra $\g$ has a compatible norm $b$ that makes $-\ad^2_v\vrule_{\,\p}$  dissipative for each $v\in\p$. We say that $(G,\tau)$ satisfies SNC (semi-negative curvature). According to Proposition \ref{criterio}, this last condition is equivalent to the fact that $M=G/K$ is a Banach-Finsler manifold with spray of semi-negative curvature.

Combining Neeb's result on the polar map (Theorem \ref{fibras}) with Theorem \ref{spliteo}, we obtain polar decompositions relative to reductive submanifolds.

\begin{coro}\label{cprgruposplit}
Let $C=q(e^{\s})$ be an expansive reductive submanifold of a Cartan-Hadamard homogeneous space $M=G/K$. Then the map 
$$
(q(e^s),s',k)\mapsto e^se^{s'}k
$$
induces an isomorphism $C\times \s'\times K\simeq G$. 

\smallskip

Assume that $C$ is also reflective. If we put $C'=q(e^{\s'})$, then $C'$ is a reductive submanifold of $M$ and we obtain an isomorphism
$$
G \simeq  C\times C'\times K.
$$
If $g=e^se^{s'}k\in G$, then $\|s\|_{\p}=d(q(g),C')$. Moreover $\|s'\|_{\p}=d(q(g),C)$ if and only if $\|1-p\|=1$, i.e. if and only if $C'$ is also expansive.
\end{coro}

\subsection{Positive elements}\label{concrete}

For a symmetric Banach-Lie group $(G,\sigma)$ one has the natural involution $^*:G\to G$ given by $g^*=\sigma(g^{-1})=\sigma(g)^{-1}$. It allows to write down the quotient map in a concrete way as
$P:G\to G$, $P(g)=gg^*$ (note that the isotropy of $1\in G$ is just $K=$ the fixed-point set of $\sigma$). Thus $M:=P(G)\simeq G/K$ has a natural structure of Finsler manifold with spray, under the usual hypothesis.

The set $P(G)$ is the set of positive invertible elements when $G$ is one of the so called classical Banach-Lie groups (see the Appendix). In this picture, the geodesics of $M$ are given by
$$
\alpha(t)=e^v e^{2tz}e^v.
$$
Let $G^s$ stand for the set of invertible self-adjoint elements, $g^*=g$, that is
$$
G^s=\{g\in G: \sigma(g)=g^{-1}\}.
$$
Then the natural action of $G$ on $G^s$ is $a\mapsto gag^*$, and if $G=\bh^{\times}$ is the subgroup of invertible elements of $\bh$ (the bounded linear operators on a Hilbert space ${\mathcal H}$), then this action defines Banach homogeneous spaces $G^{s,a}$, the orbits of $a\in G^s$; the existence of smooth local sections is essentially given by the square root of $\bh$, see \cite[Prop. 1.1]{cpr1} for the details. Via polar decomposition, one has the projection $\pi:G^s\to K^s$, where $K^s$ is the set of reflections of $G$, i.e. the set of self-adjoint elements of $K$. That is, write $g=e^vk$ where $v\in \p$ and $k\in K$ (see Theorem \ref{fibras}), and for $g\in G^s$, put $\pi(g)=k$. If $G=\bh^{\times}$, this fibration $\pi$ has very nice properties, for instance its differential is a contraction \cite[Th. 5.1]{cpr1}, a fact related to  the geodesic structure of the group of reflections $K^s$.

\section{Appendix: examples and applications}\label{appendix}

Here we indicate some applications to operator theory. We concentrate on operators ideals, and we omit other relevant examples such as bounded symmetric domains and $JB^*$-algebras. See \cite[Section 6]{neeb} for further discussion on these topics. An account of the applications for semi-finite von Neumann algebras  as studied in \cite{al} in the Riemannian situation, will be subject of a future publication.

\subsection{Operator algebras}\label{operalg}

Let $\bh$ stand for the set of bounded linear operators on a separable complex Hilbert space ${\cal H}$, with the uniform norm denoted by $\|\cdot\|$. Let $\|\cdot\|_{\cal I}:\bh \to \mathbb R_+ \cup \{\infty\}$ be an unitarily invariant norm, $\|uxv\|_{\cal I}=\|x\|_{\cal I}$ for unitary $u,v\in \bh$. Let ${\cal I}$ stand for the set of operators with finite norm, that is
$$
{\cal I}=\{x\in \bh: \|x\|_{\cal I}<\infty\}.
$$
Further one asks that 
\begin{enumerate}
\item $\|xyz\|_{\cal I}\le \|x\|\; \|y\|_{\cal I}\; \|z\|$ for any $y\in {\cal I}$ and $x,z\in \bh$.
\item  $\left({\cal I},d_{\cal I}\right)$ is a complete metric space, where $d_{\cal I}(x,y)=\|x-y\|_{\cal I}$.
\end{enumerate}

Then ${\cal I}$ is a complex self-adjoint ideal of compact operators in $\bh$, the standard reference on the subject is the book of Gohberg and Krein \cite{gk}. If $y\mapsto y^*$ denotes the usual involution of $\bh$, then it is easy to check $\|y^*\|_{\cal I}=\|y\|_{\cal I}$ and further, that the norm is unitarily invariant in the sense that
$$
\|uyv\|_{\cal I}=\|y\|_{\cal I}
$$
for any $y\in {\cal I}$ and $u,v\in \bh$ unitary operators.

\begin{rem}
The elementary examples are given by the Schatten ideals $\bp$ of operators, defined by the $p$-norms in $\bh$ ($1\le p<\infty$) by
$$
\|v\|_p^p=tr |v|^p=tr( (v^*v)^{\frac{p}{2}} ),
$$
where $tr$ is the infinite trace of $\bh$. Elements of $\bp$ are compact operators whose spectra is $l^p$ summable. One has the inequalities
$$
\|v\|\le \|v\|_p\le \|v\|_q \le \cdots\le \|v\|_1
$$
for $p\ge q$, and the inclusions
$$
{\cal B}_1({\cal H})\subset \cdots \subset {\cal B}_q({\cal H})\subset \bp \subset \dots \subset {\cal K}({\cal H}),
$$
where ${\cal K}({\cal H})$ denotes the ideal of compact operators. The trace map $(v,w)\mapsto tr(vw^*)$ induces the duality $\bp^*={\cal B}_q({\cal H})$ for $1/p+1/q=1$ and $1<p<\infty$. Moreover, ${\cal K}({\cal H})^*={\cal B}_1({\cal H})$ and ${\cal B}_1({\cal H})^*=\bh$. The ${\cal B}_p({\cal H})$ spaces are $2$-uniformly convex for $p\in (1,2]$ and $p$-uniformly convex for $p\in [2,+\infty)$, due to McCarthy's inequalities \cite{mc}.
\end{rem}

Let $G^{\cal I}$ stand for the group of invertible operators in the unitized ideal, that is
$$
G^{\cal I}=\{1+x:\; x\in {\cal I}, \;Sp(1+x)\subset \mathbb R^*\},
$$
where $Sp$ denotes the usual spectrum of an element in $\bh$. Equivalently
$$
G^{\cal I}=\{g\in \bh^{\times}: \;g-1\in {\cal I}\}.
$$
Then $G^{\cal I}$ is a Banach-Lie group (one of the so-called classical Banach-Lie groups \cite{harpe}), open in ${\cal I}$ with the inherited topology, and ${\cal I}$ identifies with its Banach-Lie algebra: it suffices to prove that a neighborhood of $1\in G^{\cal I}$ is isomorphic to ${\cal I}$. To prove these statements consider the usual analytic logarithm: for $\|g-1\|_{\cal I}<1$ put $\log(g)=\sum_k (1-g)^n$. Then if $g\in G^{\cal I}$ is such that $\|g-1\|_{\cal I}<1$, $x=log(g)\in {\cal I}$ and $e^x=g$.

\medskip

Let ${\cal I}_h$ stand for the set of self-adjoint elements in ${\cal I}$, and consider $M^{\cal I}$ the cone of positive invertible elements in the unitized ideal:
$$
M^{\cal I}=\{1+x:\; x\in {\cal I}_h, \;Sp(1+x)\subset (0,+\infty)\}.
$$

Consider the involutive automorphism $\sigma:G^{\cal I} \to G^{\cal I}$ given by $g\mapsto (g^{*})^{-1}$. Let $U^{\cal I}\subset G^{\cal I}$ stand for the unitary subgroup of fixed points of $\sigma$. Its Banach-Lie algebra is the set of skew-hermitian elements of ${\cal I}$, and ${\cal I}={\cal I}_h\oplus i {\cal I}_h$. The quotient space $G^{\cal I}/ U^{\cal I}$ can be identified with $M^{\cal I}$ via $q:G^{\cal I}\to M^{\cal I}$ given by $q(g)=gg^*$ as in Section \ref{concrete}. 

We claim that the unitarily invariant norm of ${\cal I}$ makes of $(G^{\cal I},\sigma)$ a SNC group. We use the criteria of Proposition \ref{criterio}:
$$
\|e^{i\ad x}v\|_{\cal I}=\|Ad_{e^{ix}}v\|_{\cal I}=\|e^{ix}ve^{-ix}\|_{\cal I}=\|v\|_{\cal I}
$$
for any $x,v\in {\cal I}_h$, and then $1- it\ad x$ is expansive and invertible for any $t>0$, hence $1+t\ad^2_x$ is expansive and invertible for any $t>0$, proving that $-\ad^2_x$ is dissipative for any $x\in {\cal I}_h$.

Thus the positive cone $M^{\cal I}\simeq G^{\cal I}/ U^{\cal I}$ can be regarded as a complete manifold of semi-negative curvature, since it is geodesically complete. Moreover, since $Z({\cal I})=\{0\}$ for a proper ideal ${\cal I}$, then $M^{\cal I}$ is simply connected and $\exp: {\cal I}_h\to M^{\cal I}$ is an isomorphism.

The unique geodesic of $M^{\cal I}$ joining positive invertible $a,b\in M^{\cal I}$ is short and is given by 
$$
\gamma_{a,b}(t)=a^{\frac12}(a^{-\frac12}ba^{-\frac12})^t a^{\frac12}.
$$
Its length is given by $\|\ln(a^{-\frac12}ba^{-\frac12})\|_{\cal I}$, and the exponential map at $a\in M^{\cal I}$ is given by
$$
\exp_a(v)=a^{\frac12}\exp(a^{-\frac12}va^{-\frac12}) a^{\frac12},
$$
whenever $v\in T_aM^{\cal I}\simeq {\cal I}_h$. In particular
$$
d(a,b)=\|\ln(a^{\frac12}b^{-1}a^{\frac12})\|_{\cal I}.
$$
The semi-negative curvature condition is the (well-known for matrices, see for instance \cite{bhatia}) \textit{exponential metric increasing property}
$$
\|\ln(a^{-\frac12}b a^{-\frac12})\|_{\cal I}\ge \|\ln(a)-\ln(b)\|_{\cal I}.
$$
The convexity of the geodesic distance among two geodesics starting at $a=1$ apparently is  given by the inequality
$$
\|\ln(a^{-\frac{t}{2}}b^{t}a^{-\frac{t}{2}})\|_{\cal I}\le t\|\ln(a^{-\frac12}b a^{-\frac12})\|_{\cal I}.
$$
This inequality seems to be new in this context, but for $\bh$ was extensively studied and it is known as one of the equivalent forms of Löwner-Heinz theorem on monotone operator maps \cite{lowner}. For the $p$-norms of $\bh$ it is stated as
$$
tr((B^{\frac12} A B^{\frac12})^{rp})\le tr((B^{\frac{r}{2}} A^r B^{\frac{r}{2}})^p),\quad r\ge 1,
$$
an inequality due to Araki \cite{araki}. As it is, it was generalized to the noncommutative $L^p(M,\tau)$-spaces of a semi-finite von Neumann algebra $M$ by Kosaki in \cite{kosa2}. In the context of the uniform norm, the relation between this inequality and the convexity of the geodesic distance in the positive cone of $\bh$ was studied in \cite{acs}. 

When ${\cal I}=\bp$, and $p>1$, we can apply the results of Section \ref{uniconvex} to convex closed sets. In particular, if $C\subset M$ is a convex submanifold, one obtains splittings as in Corollary \ref{spliteometrico}. These examples were studied for $p=2$ (the Riemann-Hilbert situation) in \cite{larro,tumpach}. The non-uniformly convex situation, when $p=1$, was studied in \cite{conde}. 

\medskip

In this setting, the standard example of convex submanifold is given by $C=q(e^{\s})$, where $\s$ equals the real Banach-Lie algebra of self-adjoint diagonal operators (relative to a fixed orthonormal basis $\{e_i\}$ of ${\cal H}$), and $\s'$ are the co-diagonal self-adjoint operators.

\begin{rem}\label{otras}
If we decompose a tangent vector $v\in {\cal I}_{h}$, $v=w+z$ and $\|v\|_{\cal I}=\|w\|_{\cal I}+\|z\|_{\cal I}$, then the curve
$$
\delta(t)=\left\{
\begin{array}{lll}    
e^{2t w}  & & t\in [0,1/2]\\
\\
e^we^{(2t-1)z} & & t\in [1/2,1]
\end{array}\right.
$$
is piecewise smooth and joins $1$ to $e^v$ in $M_{\cal I}$; moreover $L(\delta)=\|w\|_{\cal I}+\|y\|_{\cal I}=L(\exp(tv))$, so $\delta$ is a minimizing piecewise smooth curve joining $1$ to $e^v$, and it is not a geodesic unless $w$ and $z$ are aligned. This shows that Proposition \ref{uni} is false for $p=1$ and $p=\infty$ (whose norm is not strictly convex). For example, consider $v=\frac12 p_1+\frac12(p_1+ p_2)$ with $p_i$ mutually orthogonal one dimensional projections in $\bh$. Then $x=\frac12 p_1$ and $y=\frac12(p_1+ p_2)$ commute, and $\|v\|_{\infty}=1=1/2+1/2=\|x\|_{\infty}+\|y\|_{\infty}$.
\end{rem}

\subsection{Inclusions of $C^*$-algebras}\label{inclusion}

Let ${\cal A}\subset \bh$ be a $C^*$-subalgebra, and let ${\cal E}:\bh\to {\cal A}$ be a conditional expectation with range ${\cal A}$. Let $H$ stand for the linear supplement of ${\cal A}$ given by ${\cal E}$, that is $H=\ker {\cal E}$. Then $\|{\cal E}\|=1$ and ${\cal E}$ is a bi-module map, that is ${\cal E}(nmn')=n{\cal E}(m)n'$ for any $n,n'\in {\cal A}$.

In \cite{pr}, the authors studied inclusions of $C^*$-algebras $N\subset M$ with a conditional expectation ${\cal E}:M\to N$. In that settting, one has the inclusions $P_N\subset P_M$ of cones of positive invertible elements; the tangent spaces are the sets of self-adjoint elements of $N$ and $M$ respectively. The projection $p={\cal E}\vrule_{M_h}:M_h\to N_h$ provides a reductive supplement $H=\ker p$ for $N_h$, and moreover $\|p\|=1$. The exponential map provides a splitting of the positive cone $P_M$ of $M$ via the positive cone $P_N$ of $N$ as a convex submanifold, and $H$ as the normal bundle. In such a situation, the norm of $1-{\cal E}$ can be as large as $2$. The purpose of this short section is to extend this situation to the Finsler norms of the $p$-Schatten ideals, applying the results of the previous sections.

\medskip

Let $p\ge 1$ and put ${\cal A}_p={\cal A}\cap \bp$, and ${\cal E}_p={\cal E}\vrule_{\bp}$. In certain situations one can ensure that ${\cal E}({\cal B}_1({\cal H}))\subset {\cal B}_1({\cal H})$. A sufficient condition is that ${\cal E}$ maps finite rank operators into finite rank operators, a condition which is easy to check in most situations. Throughout, it is assumed that the expectation is compatible with the trace, that is $Tr({\cal E}x)=Tr(x)$ for any $x\in {\cal B}_1({\cal H})$. The example to have in mind is that of a maximal abelian subalgebra ${\cal A}$ given by the diagonal operators in some fixed orthonormal base of ${\cal H}$. In this case the conditional expectation is given by compression to the diagonal.

\medskip

Note that by duality (since $\|{\cal E}\|=1$)
\begin{eqnarray}
\|{\cal E}_1(x)\|_1 & = & \sup\limits_{\|z\|\le 1}|tr({\cal E}(x)z)| = \sup\limits_{\|z\|\le 1}|tr({\cal E}(x){\cal E}(z))|= \sup\limits_{\|w\|\le 1, w\in {\cal A}}|tr({\cal E}(x)w)| \nonumber\\
& = & \sup\limits_{\|w\|\le 1,w\in {\cal A}}|tr(xw)|\le \sup\limits_{\|w\|\le 1}|tr(xw)|=\|x\|_1.\nonumber
\end{eqnarray}
Thus $\|{\cal E}_1\|\le 1$, and since ${\cal E}_1$ is a projection, $\|{\cal E}_1\|=1$. The essence of the argument is the fact that ${\cal E}$ (as a Banach space linear operator) is self-dual. Then $1-{\cal E}$ is also self-dual, and with the same proof, one also has $\|1-{\cal E}_1\|\le \|1-{\cal E}\|$.

From the fact that $\bp$ can be obtained via complex interpolation from the pair $({\cal B}_1({\cal H}),{\cal B}({\cal H}))$ (see for instance \cite{simon}), and that ${\cal B}_1({\cal H})$ is dense in each $\bp$ (since finite rank operators are dense), it follows that the restriction ${\cal E}_p$ defined above matches the interpolated conditional expectation.

Now we observe that for $p=2$, this restriction is an orthogonal projection: indeed,
\begin{eqnarray}
\|{\cal E}_2(z)\|_2^2 &= & Tr( {\cal E}(z){\cal E}(z)^*)=Tr ({\cal E}(z){\cal E}(z^*))=Tr(z^*{\cal E}(z))\nonumber\\
&\le& Tr(z^*z)^{\frac12} Tr({\cal E}(z){\cal E}(z)^*)^{\frac12}=\|z\|_2\|{\cal E}_2(z)\|_2\nonumber
\end{eqnarray}
by the Cauchy-Schwarz inequality for the trace inner product. With the same argument, $\|1-{\cal E}_2\|=1$. 

From these facts (using interpolation again) follow that, for any $p\in [1,2]$,
$$
\|{\cal E}_p\|=1\; \mbox{ and }\; \|1-{\cal E}_p\|\le \|1-{\cal E}\|^{\frac{2}{p}-1}.
$$
Then by duality,
$$
\|{\cal E}_p\|=1\; \mbox{ and }\; \|1-{\cal E}_p\|\le \|1-{\cal E}\|^{1-\frac{2}{p}}
$$
holds for any $p\in [2,\infty)$.

\medskip

Certainly $[\cal A,\cal A]\subset A$ since ${\cal A}$ is an associative subalgebra, so evidently $\ad_{{\cal A}_h}^2({\cal A}_h)\subset {\cal A}_h$. But note also that, since ${\cal E}$ is a bi-module map, that $[\cal A,\ker {\cal E}]\subset \ker {\cal E}$. So $C=\exp({\cal A}_h)$ is a reductive submanifold of the positive cone of ${\cal B}({\cal H})$. Hence by restricting the conditional expectation to the self-adjoint part of the $p$-Schatten ideals one sees that the positive cone of the (unitized) subalgebra ${\cal A}_p$ has a natural structure of reductive expansive submanifold in $\bp$. Thus one obtains splittings of the respective classical Banach-Lie groups invoking Corollary  \ref{cprgruposplit}:
\begin{teo}
Let ${\cal A}\subset \bh$ be a $C^*$-algebra with a conditional expectation ${\cal E}:\bh\to {\cal A}$ compatible with the trace, such that ${\cal E}({\cal B}_1({\cal H}))\subset {\cal B}_1({\cal H})$. Let $p\ge 1$, and let ${\cal A}_p={\cal A}\cap \bh$. Then for each invertible element $g$, 
$$
g\in G_p({\cal H})=\{g\in \bh^{\times}:g-1\in\bp\}
$$
there exists unique operators $u,g_{\cal A},v_p$ such that
\begin{enumerate}
\item $u$ is a unitary operator and
$$
u\in U_p({\cal H})=\{u\in {\cal U}({\cal H}):u-1\in\bp\},
$$
\item $g_{\cal A}$ is invertible and
$$
g_{\cal A}\in {\cal A}_p^{\times}=\{g\in \bh^{\times}:g-1\in {\cal A}_p\},
$$
\item $v_p\in \bp_h$ and ${\cal E}(v_p)=0$,
\item the operator $g$ can be decomposed as
$$
g=g_{\cal A} e^{v_p} u,
$$
which gives the isomorphism
$$
G_p({\cal H})\simeq {\cal A}_p^{\times}\times \left( \ker {\cal E}\cap \bp_h \right) \times U_p(\cal H).
$$
\end{enumerate}
If $\|1-{\cal E}\|=1$, then $\|v_p\|_p=d(\sqrt{gg^*},{\cal A}_p^+)$, where $d$ indicates the geodesic distance in the positive cone, and ${\cal A}_p^+$ denotes the positive cone of the (unitized) subalgebra ${\cal A}_p$ of $\bp$. Equivalently, if we write $g=e^v$ with $v\in \bp_h$ and $g_{\cal A}=e^{Z_{\cal A}}$ with $Z_{\cal A}$ in the self-adjoint part of ${\cal A}_p$, then $Z_{\cal A}$ is the unique minimizer of the nonlinear functional $\varphi: ({\cal A}_p)_h\to\mathbb R_+$ given by
$$
\varphi: z\mapsto \|\ln(e^{v/2}e^{-z}e^{v/2})\|_p.
$$
\end{teo}

\medskip

These factorizations, in the context of $n\times n$ real matrices, for the Riemannian metric induced by the trace, stem back to Mostow \cite{mostow}, where he uses the semi-paralellogram laws to obtain the best approximant, bringing new light on the real linear group.

%NOTA: La geodésica de positivos que empieza en $a$ y termina en $b$ está dada por $g=a^{\frac12}$, $v=\frac12 \ln(a^{-\frac12}ba^{-\frac12})$, donde $q(x)=xx^*$ y además $\dot{\alpha}(0)=q_{*g}(gv)=2gvg^*$. 

\bigskip

\noindent
Cristian Conde and Gabriel Larotonda\\
Instituto de Ciencias \\
Universidad Nacional de General Sarmiento \\
J. M. Gutierrez 1150 \\
(B1613GSX) Los Polvorines \\
Buenos Aires, Argentina  \\
e-mails: cconde@ungs.edu.ar, glaroton@ungs.edu.ar

\end{document}